\documentclass[reqno,12pt]{amsart}
\usepackage[latin1]{inputenc}
\usepackage{amsmath,amsthm,amssymb,amscd,pb-diagram}

\newcommand{\df}{\bf\em}
\newcommand{\wt}{\widetilde}
\newcommand{\wh}{\widehat}
\newcommand{\Ac}{\alpha}
\newcommand{\ApX}{\alpha+\chi}
\newcommand{\DaX}{\alpha+\chi}
\newcommand{\AAA}{\alpha}

\newcommand{\comment}[1]{}
\newcommand{\G}[1]{\mathfrak{G}_{#1}}
\newcommand{\Gk}{{\G{k}}}
\newcommand{\Gp}{{\G{\pp}}}
\newcommand{\gprod}{\times}
\newcommand{\set}[1]{\{#1\}}
\newcommand{\gen}[1]{\l<#1\r>}
\newcommand{\field}[1]{\mathbb{#1}}
\newcommand{\Q}{\field{Q}}
\newcommand{\C}{\field{C}}
\newcommand{\N}{\field{N}}
\newcommand{\Z}{\field{Z}}
\newcommand{\F}{\field{F}}
\newcommand{\Ps}{\field{P}}
\newcommand{\ift}{\infty}
\newcommand{\ndiv}{\nmid}
\newcommand{\pp}{{\mathfrak p}}
\newcommand{\PP}{{\mathfrak P}}
\newcommand{\QQ}{{\mathfrak Q}}
\newcommand{\lra}{\longrightarrow}
\newcommand{\bt}{{\mathbf t}}
\newcommand{\kt}{k(\bt)}
\newcommand{\ktt}{k(\!(\bt)\!)}
\newcommand{\Fptt}{\F_p(\!(\bt)\!)}
\newcommand{\sett}[2]{\{#1\,|\,#2\}}
\newcommand{\Sett}[2]{ \left\{ {#1}\, \left| \,{#2} \right.\right\} }
\newcommand{\SetL}[2]{ \left\{ \left. {#1}\, \right| \,{#2} \right\} }

\newcommand{\hp}[1]{\operatorname{ht}_{#1}}     % \hp = additive p-height
\DeclareMathOperator{\l<}{\langle}
\DeclareMathOperator{\r>}{\rangle} 
\DeclareMathOperator{\Gal}{Gal}
\DeclareMathOperator{\SA}{\underline{supp}}
\DeclareMathOperator{\Br}{Br} 
\DeclareMathOperator{\ind}{ind}
\DeclareMathOperator{\Hom}{Hom} 
\DeclareMathOperator{\supp}{supp} \DeclareMathOperator{\inv}{inv}
\DeclareMathOperator{\charak}{char} \DeclareMathOperator{\ord}{ord}
\DeclareMathOperator{\res}{res} 
\DeclareMathOperator{\img}{im} 

\newtheorem{theorem}{Theorem}
\newtheorem{lemma}[theorem]{Lemma}
\newtheorem{prop}[theorem]{Proposition}
\newtheorem{cor}[theorem]{Corollary}
\theoremstyle{remark}
\newtheorem*{ack}{Acknowledgements}
\theoremstyle{definition}
\newtheorem{defn}[theorem]{Definition}
\newtheorem*{rem*}{Remark}
\newtheorem{ex}[theorem]{Example}
\newtheoremstyle{citing}% name
  {3pt}%      Space above, empty = \upsilonsual value'
  {3pt}%      Space below
  {\itshape}% Body font
  {}%         Indent amount (empty = no indent, \parindent = para indent)
  {\bfseries}% Thm head font
  {.}%        Punctuation after thm head
  {.5em}%     Space after thm head: `` '' = normal interword space;
  {\thmnote{#3}}% Thm head spec
\theoremstyle{citing}
\newtheorem*{varthm}{}
\numberwithin{theorem}{section}
\numberwithin{equation}{section}

\begin{document}

\author{Timo Hanke}
\address{
Department of Mathematics\\
Technion --- Israel Institute of Technology\\
Haifa, 32000\\
Israel }
\email{hanke@math.rwth-aachen.de}
\curraddr{
Lehrstuhl D für Mathematik\\
RWTH Aachen\\
Templergraben 64\\
D-52062 Aachen\\
Germany} 
\thanks{The first author was supported by the Technion --- Israel
Institute of Technology and by the 5th Framework Programme of the
European Commission (HPRN-CT-2002-00287)}

\author{Jack Sonn}
\address{
Department of Mathematics\\
Technion --- Israel Institute of Technology\\
Haifa, 32000\\
Israel }
\email{sonn@tx.technion.ac.il}
\thanks{The second author was supported by the Fund for Promotion of Research at the Technion}

\title{Noncrossed products in Witt's Theorem}
\date{\today}
 %Twisted and skew group rings, crossed products

\begin{abstract}
Since Amitsur's discovery of noncrossed product division algebras
more than 35 years ago, their existence over more familiar fields
has been an object of investigation. Brussel's work was a
culmination of this effort, exhibiting noncrossed products over the
rational function field $\kt$ and the Laurent series field $\ktt$
over any global field $k$ --- the smallest possible centers of
noncrossed products.

Witt's theorem gives a transparent description of the Brauer group
of $\ktt$ as the direct sum of the Brauer group of $k$ and the
character group of the absolute Galois group of $k$.
We classify the Brauer classes over $\ktt$ containing noncrossed
products by analyzing the fiber over $\chi$ for each character
$\chi$ in Witt's theorem.
In this way,
a picture of the partition of the Brauer group into crossed products/noncrossed products is obtained,
which is in principle ruled solely by a relation between index and number of roots of unity.
For large indices the noncrossed products occur
with a ``natural density'' equal to $1$.
\end{abstract}
\maketitle 
\section{Introduction}
A finite-dimensional division algebra is called a {\em crossed
product} if one of its maximal commutative subfields is Galois over
the center of the algebra, otherwise a {\em noncrossed product}.
The existence
of noncrossed products was for several decades the biggest open
question in the theory of finite-dimensional division algebras
before it was settled by Amitsur \cite{amitsur:central-div-alg} in
1972.

The present paper is motivated by the work of Brussel [5,6]
\nocite{brussel:noncr-prod} \nocite{brussel:noncr-prod-ff} on the
existence of noncrossed products over the fields $\kt$ and $\ktt$
for any global field $k$, which determined these fields to be the
``smallest'' possible centers of noncrossed products. 

Brussel's
result is even more surprising over $\ktt$ than over $\kt$ because of the much
simpler structure of its Brauer group. Due to {\em Witt's theorem}
\cite{witt:schiefkoerper}, if $k$ is perfect, there is a canonical
isomorphism
\begin{equation}\label{eq:witt}\Br(\ktt)\cong\Br(k)\oplus X(\Gk)\end{equation}
(depending on $\bt$),
where $X(\Gk)=\Hom_c(\Gk,\C^*)$ denotes the group of continuous characters of the absolute Galois group of~$k$.
Each element $\chi$ of the torsion group $X(\Gk)$ belongs to a finite cyclic extension field $k(\chi)\supseteq k$
with degree equal to the order $|\chi|$ of $\chi$.
($k(\chi)$ is the fixed field of $\ker\chi$.) 
Moreover, a canonical generator $\sigma_\chi$ of the Galois group of $k(\chi)$ over $k$
is obtained by restricting to $k(\chi)$ any $\sigma\in \Gk$ with $\chi(\sigma)=e^{2\pi i/|\chi|}$.
Under \eqref{eq:witt}, $\chi$ maps to the class of the cyclic algebra
$(k(\chi)(\!(\bt)\!)/\ktt,\sigma_\chi,\bt)$ over $\ktt$. 

If $k$ is
non-perfect then $\Br(k)\oplus X(\Gk)$ is a subgroup  of
$\Br(\ktt)$, the so-called ``tame'' part. Throughout the paper we
will work exclusively inside this subgroup, and the global field $k$
may be perfect or not. 

Let $k$ be a global field, i.e.\ either a number field or a function
field in one variable over a finite field.
Brussel proves the existence of pairs of $\alpha\in\Br(k)$ and
$\chi\in X(\Gk)$ 
such that the $\ktt$-division algebra representing $\alpha+\chi$
is a noncrossed product.
In this way, all indices for which noncrossed products are known to exist
can be realized. 
However,
it has not yet been determined for all classes  $\alpha+\chi$
whether the representing division algebra is a noncrossed product.
This task is the motivation for the present paper. 

Our approach is
to partition $\Br(k)\oplus X(\Gk)$ into the union of fibers over
$\chi$ where $\chi$ runs through $X(\Gk)$ and to classify the
noncrossed products in each fiber.
(Division algebras are identified with their respective Brauer classes.)

Let $\chi\in X(\Gk)$ and consider all division algebras of fixed  index $N\in\N$ inside $\Br(k)+\chi$.
We will show that one of the following two cases occurs:
\begin{enumerate}
  \item[(I)] All division algebras in the fiber over $\chi$ of index $N$ are crossed products.
  \item[(II)] Among all division algebras in the fiber over $\chi$ of index $N$
    the noncrossed products have a ``natural density'' equal to $1$.
    In particular, there are infinitely many.
\end{enumerate}
Moreover, for fixed $\chi$ and varying $N$ the cases (I) and (II) are separated by bounds on the 
prime powers dividing $N$ in such a way that, roughly speaking,
case (I) occurs ``below'' the bounds and case (II) ``above''.
The details are as follows.
By a well-known formula (cf.\ \eqref{eq:ind} below)
all elements in $\Br(k)+\chi$ have index a multiple of $|\chi|$.
Let $N=|\chi|m$ and let $\prod p^{n_p}$ be the prime factorization of $m$.
There are simply defined bounds $b_p(\chi)$ where $p$ runs through the prime factors of $|\chi|$,
each $b_p(\chi)$ a nonnegative integer or infinity, such that we are in case
\begin{enumerate}
  \item[(I)] if $n_p\leq b_p(\chi)$ for all prime factors $p$ of $|\chi|$,
  \item[(II)] if $n_p> b_p(\chi)$ for some prime factor $p$ of $|\chi|$.
\end{enumerate}

Analogous results hold over $\kt$ as well as over $\ktt$.
Furthermore, the field $\ktt$ can be replaced by any discrete
rank one Henselian valued field with residue field~$k$.

\begin{ack}
The first author would like to thank the Departments of Mathematics
at Bar-Ilan University and at the Technion for their kind
hospitality and support during his visits in 2006 and gratefully
acknowledges the financial support from the European Commission and
the Technion.
\end{ack}

\section{Notation and Statement of Results}\label{sec:results}

Let $k$ be a field and $k_{sep}$ its separable closure.
We write $\Gk$ for $\Gal(k_{sep}/k)$, the absolute Galois group of $k$. 
The nonnegative integer $s_p(k)$, $p$ prime,
is defined by the condition that $\prod p^{s_p(k)}$ is the number of roots of unity contained in $k$.
(If $p=\charak k$ then $s_p(k)=0$.)

Denote by $X(\Gk)$ the group $\Hom_c(\Gk,\C^*)$ of continuous characters.
($\Gk$ is a pro-finite group with the Krull topology and $\C^*$ is equipped with the discrete topology.) 
$X(\Gk)$ is an abelian torsion group. 
For each $\chi\in X(\Gk)$, the fixed field of $\ker\chi$
is a cyclic extension of $k$ of degree $|\chi|$
that we shall denote by $k(\chi)$. 
We abbreviate $s_p(k(\chi))$ to $s_p(\chi)$.

Since $X(\Gk)$ is an abelian torsion group,
each of its $p$-primary components $X_p(\Gk)$ is a direct summand.
We let $\chi_p$ denote the projection of $\chi\in X(\Gk)$ onto $X_p(\Gk)$.
As an element of the abelian $p$-group $X_p(\Gk)$, $\chi_p$ has a {\em height}, 
which is defined to be the maximal nonnegative integer $r$ such that $\chi_p$ is divisible by
$p^r$, or infinity if no maximal $r$ exists 
(cf.\ Kaplansky \cite{kapl:inf-ab-grps}, \S9, p.19).
We denote the height of $\chi_p$ by $\hp{p}(\chi)$.

In terms of the corresponding cyclic field extensions,
$k(\chi_p)$ is the unique subfield of $k(\chi)$ with maximal $p$-power degree over $k$.
Moreover, $\hp{p}(\chi)$ is the
maximal nonnegative integer $r$ such that the cyclic extension $k(\chi_p)/k$ embeds 
into a cyclic extension $L/k$ with $[L:k(\chi_p)]=p^r$,
or infinity if no maximal $r$ exists. 
Since $\hp{p}(\chi)$ is an invariant of the field extension $k(\chi)/k$,
the notation $\hp{p}(k(\chi)/k)$ is also valid.

\begin{defn}\label{def:except}
Let $\chi\in X(\Gk)$.\\
a) $\chi$ is said to be {\df exceptional} 
if $k$ is a number field, $i:=\sqrt{-1}\in k(\chi)$ and
\begin{equation*}
 \hp{2}(k(\chi)/k(i))>\hp{2}(k(\chi)/k)>0.
\end{equation*}
b) For each prime $p$ define
\begin{equation*}
b_p(\chi):=\begin{cases}\hp{p}(\chi)+s_p(\chi)+1 &\textrm{if $p=2$ and $\chi$ is exceptional,}\\
\hp{p}(\chi)+s_p(\chi) &\textrm{otherwise.}
\end{cases}
\end{equation*}
Being exceptional as well as the numbers $b_p(\chi)$ are all invariants of the field extension $k(\chi)/k$.
Therefore, we can speak of an {\em exceptional extension} $k(\chi)/k$ 
and we can write $b_p(k(\chi)/k)$.
\end{defn}

Clearly, $\chi$ is exceptional if and only if $\chi_2$ is exceptional.

The extension $k(i)/k$ is exceptional if and only if $k$ is a number field, $i\not\in k$ and $\infty>\hp{2}(k(i)/k)>0$.
This is the case for all $k=\Q(\sqrt{-2a})$ with $a$ a positive integer $\equiv 7\pmod{8}$
(cf.\ Geyer-Jensen \cite[p.\ 713]{geyer-jensen}).

\begin{rem*}
It can be shown that exceptional characters exist also with $K\supsetneq k(i)$ and with $s_2(\chi)$ taking all values $\geq 2$.
Details will appear in a separate paper \cite{hanke:height}. 
\end{rem*}

\begin{theorem}[Main Theorem]\label{thm:x}
Let $k$ be a global field and let $\chi\in X(\Gk)$  be a character.
The division algebras in the fiber $\Br(k)+\chi$ all have index $|\chi|m$
  for some positive integer $m$.
Let $m$ be fixed and let $m=\prod p^{n_p}$ be its prime factorization.

(i) If $n_p\leq b_p(\chi)$ for all prime factors $p$ of $|\chi|$
  then all division algebras in the fiber over $\chi$ of index $|\chi|m$
  are crossed products.

(ii) If $|\chi|$ has a prime divisor $p$ with $n_p>b_p(\chi)$ then 
  the fiber over $\chi$ contains noncrossed product division algebras
  of index $|\chi|m$.
  Moreover, among all division algebras in the fiber over $\chi$ of index $|\chi|m$
  the noncrossed products have a ``natural density'' equal to~$1$.
\end{theorem}

For  algebras of $p$-power index ($p$ prime) the picture becomes particularly simple:
if $|\chi|$ and $m$ are both powers of $p$ then 
we are in case 
\begin{equation*} 
  \text{(I) if $n_p\leq b_p(\chi)$, \quad (II) if $n_p>b_p(\chi)$.}
\end{equation*}
Since it is known that $\hp{p}(\chi)=\ift$ for $p=\charak k$,
there are no noncrossed product $p$-algebras in $\Br(k)\oplus X(\Gk)$.

The proof of Theorem \ref{thm:x} for exceptional characters $\chi$ is more difficult than
that for non-exceptional $\chi$, so the additional pieces required for the
exceptional characters are separated out and presented later, in the interest of
greater readability, after the proof for non-exceptional characters is complete.
On the other hand, there are interesting aspects to this part of the proof of the
exceptional case, 
involving non-routine applications of 
the special case of the local-global principle for the height of a cyclic extension of a number field,
and of Neukirch's theorems on embedding problems with prescribed local solutions.
The embedding problems arise for a family of metacyclic $2$-groups
which are exceptional in their own way
(in the classification of metacyclic $p$-groups).

\section{Preliminaries}\label{sec:prel}
\subsection*{Cyclotomic fields}

Let $k$ be any field.
We denote by $\mu(k), \mu_n(k)$ and $\mu_{p^\infty}(k)$ the group of all roots (resp.\ all $n$-th roots, all $p$-power roots) of unity contained in $k$.
Given that $k$ is fixed, $\mu, \mu_n, \mu_{p^\infty}$ are the respective groups of roots of unity contained in $k_{sep}$.
Hence, $k(\mu_n)$ is the $n$-th cyclotomic field over $k$.

We choose once and for all a system of primitive roots of unity $\zeta_m$ such that $\zeta_m^{n}=\zeta_{m/n}$ for all $n|m$.
We set $\eta_m:=\zeta_m+\zeta_m^{-1}$ and $i:=\sqrt{-1}$.
The fields $k(\eta_m)$ do not depend on the choice of $\zeta_m$.

For fields of characteristic zero we define $\tilde k:=k\cap\Q(\mu_{2^{\infty}})$.
Recall that \eqref{eq:diag} %Figure \ref{fig} 
below is the full subfield lattice of $\Q(\mu_{2^{\infty}})$ 
in which all lines indicate quadratic extensions.
%\begin{figure}
%\caption{}
%\label{fig}
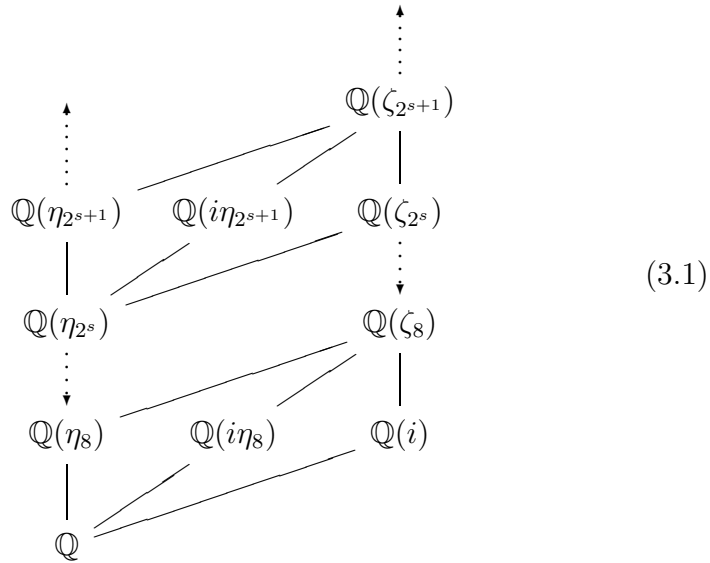
\divide\dgARROWLENGTH by2
\begin{equation}\label{eq:diag}
\begin{diagram}
  \node[3]{}\\  
  \node[3]{\Q(\zeta_{2^{s+1}})}\arrow{n,..}\\  
  \node{\Q(\eta_{2^{s+1}})}\arrow{n,..}\arrow{ene,-}\node{\Q(i\eta_{2^{s+1}})}\arrow{ne,-}\node{\Q(\zeta_{2^{s}})}\arrow{n,-}\arrow{s,..}\\
  \node{\Q(\eta_{2^{s}})}\arrow{n,-}\arrow{ne,-}\arrow{ene,-}\arrow{s,..}
  \node[2]{\Q(\zeta_8)}\arrow{s,-}\\
  \node{\Q(\eta_8)}\arrow{ene,-}\arrow{s,-}
  \node{\Q(i\eta_8)}\arrow{ne,-}
  \node{\Q(i)}\\
  \node{\Q}\arrow{ne,-}\arrow{ene,-}
\end{diagram}
\end{equation}
%\end{figure}

\subsection*{Local fields}\label{sec:lf}

By a local field we mean a finite extension of $\Q_p$ or $\Fptt$.
Let $K/k$ be a Galois extension of local fields.
We write $\overline{k}$ for the residue field.
Let $I$ denote the maximal unramified subextension of $K/k$,
$q$ the number of elements in the residue field of $k$,
and $e$ the ramification index of $K/k$.
We suppose $K/k$ is tamely ramified, i.e.\ $\charak\overline{k}\ndiv e$.
Then $K/I$ is a cyclic Kummer extension of exponent $e$
(cf.\ Lang \cite[Chapter II, \S 5, Proposition 12, p.\ 52]{lang:algnt2nd}).
In particular, $\mu_e\subset I$ and $\Gal(K/k)$ is metacyclic.
Due to Kummer theory, the action of $\Gal(I/k)$ on $\Gal(K/I)$ 
is given by the action of $\Gal(I/k)$ on $\mu_e$.
The latter lifts from the residue field, 
so the Frobenius element of $\Gal(I/k)$ acts on $\Gal(K/I)$ as the $q$-th power map.
Thus, $\Gal(K/k)$ has a presentation
\begin{equation}\label{eq:local-Gal}
\gen{x,y \,|\, x^e=1, y^f=x^t, x^y=x^q}
\end{equation}
for some $t|e$.
Here, $y|_I$ is the Frobenius
and $x$ is a certain generator of $\Gal(K/I)$.
(Note that $x$ has to be suitably chosen in order to achieve $t|e$.)
The parameters are further subject to the relations
$q^f\equiv 1\pmod{e}$ and $q\equiv 1\pmod{e/t}$
(because $y^f$ commutes with $x$ and $x^t$ commutes with $y$ respectively).
Note that
\begin{equation}\label{eq:abel-e}
  \textrm{$K/k$ is abelian }\iff\textrm{ $q\equiv 1\pmod{e}$,}
\end{equation}
in which case $K/k$ is cyclic if and only if $t=1$.

\subsection*{Global fields}

By a global field we mean a finite extension of $\Q$ or a finite extension of $\F_p(\bt)$.
Let $k$ be a global field.
For any prime $\pp$ of $k$ --- archimedian or non-archimedian ---
we write $k_\pp$ for the completion of $k$ at $\pp$.
If $F/k$ is Galois then $F_\pp$ denotes the completion of $F$
at any prime of $F$ dividing $\pp$. ($F_\pp$ is unique up to $k$-isomorphism.)

Now, let $\pp$ be a non-archimedian prime of $k$.
We denote by $N(\pp)$ the absolute norm of $\pp$,
i.e.\ the number of elements in the residue field of $\pp$.
Let $F/k$ be a finite Galois extension and let $\PP$ be a prime of $F$ dividing $\pp$.
We denote by $Z_\PP(F/k), I_\PP(F/k), e_\pp(F/k)$ and $\varphi_\PP(F/k)$ respectively the
decomposition field, the inertia field, the ramification index 
and the Frobenius element for $\PP$ 
($e_\pp$ does not depend on the choice of $\PP$, for $F/k$ is Galois).
Recall that $\Gal(I_\PP/Z_\PP)$ is generated by $\varphi_\PP$.
One says $\pp$ splits completely in $F$ if $Z_\PP(F/k)=F$ for all $\PP|\pp$,
and $\pp$ is unramified in $F$ if $I_\PP(F/k)=F$ for all $\PP|\pp$.
For instance, in $k(\mu_m)$ any $\pp$ with $(N(\pp),m)=1$ is unramified and
\begin{align}
  \label{eq:split1}
  \pp\text{ splits completely in }k(\mu_m) &\iff N(\pp)\equiv 1\pmod{m},\\
  \label{eq:split2}
  \pp\text{ splits completely in }k(\eta_m) &\iff N(\pp)\equiv \pm 1\pmod{m}.
\end{align}

Chebotarev's Density Theorem\label{chebot}
(cf.\ Weil \cite{weil:bnt3rd}, Chapter~XIII, \S12, Theorem~12, p.289)
implies that 
any given $\sigma\in\Gal(F/k)$ is the Frobenius element for infinitely many primes $\PP$ of $F$.
%  or Fried-Jarden \cite[Theorem~5.6, p.\ 58]{fried-jarden}
Applied to a compositum  $F_1F_2$ of two Galois extensions
$F_1,F_2\supseteq k$ this means that for any
$\sigma_1\in\Gal(F_1/k)$ and $\sigma_2\in\Gal(F_2/k)$ coinciding on
$F_1\cap F_2$, there are infinitely many non-archimedian primes
$\PP$ of $F_1F_2$ such that $\sigma_i$ is the Frobenius element
for $\PP\cap F_i$, $i=1,2$.

\section{The Height}\label{sec:height}

Let $K/k$ be a cyclic field extension.
In the study of $\hp{p}(K/k)$ we assume $[K:k]$ to be a $p$-power,
for $\hp{p}(K/k)$ depends by definition only on the maximal $p$-power subextension of $K/k$.
However, all results of this section formally hold for arbitrary $[K:k]$.
Before focusing on local and global fields we mention
two important facts over arbitrary fields.
The first one is known as
\begin{varthm}[Albert's theorem]
If $\mu_{p^r}\subseteq k$ then 
$$\hp{p}(K/k)\geq r \iff \mu_{p^r}\subset N_{K/k}(K^*).$$
\end{varthm}
\begin{proof}
Albert \cite[Chapter 9, Theorem 11]{albert:modern-algebra}.
\end{proof}
For instance, $-1$ is not a norm in $\Q(\sqrt{-1})/\Q$, 
so $\hp{2}(\Q(\sqrt{-1})/\Q)=0$.  
The second fact is 
\begin{varthm}[Artin-Schreier theory]
If $\charak k=p$ and $p|[K:k]$ then $$\hp{p}(K/k)=\ift.$$
\end{varthm}
\begin{proof}
Albert \cite[p.\ 194f]{albert:modern-algebra}.
\end{proof}
For local and global fields we have $\hp{p}(k/k)=\ift$.
Hence, $\hp{p}(K/k)=\ift$ for all $p$ not dividing $[K:k]$.
But infinite height is also possible in non-trivial cases and in characteristic zero:
For instance, 
$\Q\subset\Q(\sqrt{2})=\Q(\eta_8)\subset\Q(\eta_{16})\subset\ldots$
is an infinite tower of cyclic number fields of $2$-power degree,
hence $\hp{2}(\Q(\eta_{2^i})/\Q)=\ift$ for all $i$.  

\begin{prop}\label{prop:lht}
Suppose $K/k$ is a cyclic extension of local fields.
Let $e$ be the ramification index of $K/k$
and let $v_p(e)$ denote the $p$-adic exponential value of $e$.
Then
$$\hp{p}(K/k)=\begin{cases}\infty&\textrm{if $p\ndiv e$,}\\
s_p(k)-v_p(e)&\textrm{if $p|e$ and $p\neq\charak\overline k$.} 
\end{cases}$$
\end{prop}
\begin{proof}
Assume $[K:k]$ is a $p$-power.
If $p\ndiv e$ then $K/k$ is unramified, so $\hp{p}(K/k)=\infty$.
Assume now $p|e$ and $p\neq\charak\overline k$.
Let $I$ denote the maximal unramified subfield of $K$.
Write $s=s_p(k)$ and $v=v_p(e)$.

The inequality $\hp{p}(K/k)\geq s-v$ follows from Albert's Theorem.
Indeed, since $\zeta_{p^s}$ is a norm in $I/k$ and $[K:I]=e=p^v$,
$\zeta_{p^{s-v}}$ is a norm in $K/k$.
Conversely, let $L/k$ be a cyclic extension containing $K$ with $[L:K]=p^h$.
We show $h\leq s-v$.
The intermediate fields of $L/k$ are linearly ordered.
Therefore, since $I\subsetneq K$, the maximal unramified subfield of $L/k$ equals $I$.
Hence, the ramification index of $K/k$ is $p^{v+h}$.
By \eqref{eq:abel-e}, $v+h\leq s$.
\end{proof}

\begin{lemma}\label{lem:ht-min}
Suppose $K/k$ is a cyclic extension of global fields.
Then $\hp{p}(K/k)\leq \hp{p}(K_\pp/k_\pp)$
for any prime $\pp$ in $k$.
\end{lemma}
\begin{proof}
We may assume $[K:k]$ is a $p$-power.
Suppose $L/k$ is cyclic, $L\supseteq K$ and $[L:K]=p^r$.
Let $\pp$ be a prime in $k$ and claim $\hp{p}(K_\pp/k_\pp)\geq r$.
Let $Z$ be the (unique) decomposition field of $\pp$ in $L/k$.
Since the subfields of $L/k$ are linearly ordered we have $Z\supseteq K$ or $Z\subseteq K$.
In the former case $K_\pp=k_\pp$, hence $\hp{p}(K_\pp/k_\pp)=\infty$.
In the latter case $[L_\pp:K_\pp]=[L:K]$, hence  $\hp{p}(K_\pp/k_\pp)\geq r$.
\end{proof}

\begin{ex}\label{ex:ht}
If $p,q$ are primes with $q=ap^n+1$ and $(a,p)=1$ then the degree~$p^m$ 
subfield $K$ of the $q$-th cyclotomic field $\Q(\mu_q)$, $1<m\leq n$, has $\hp{p}(K/\Q)=n-m$.  
\end{ex}
\begin{proof}
Since $\Q(\mu_q)$ is cyclic of degree $q-1$, 
it contains a cyclic extension of degree $p^n$.
Thus, $\hp{p}(K/\Q)\geq n-m$.  
Conversely, the prime $q$ is totally ramified in $\Q(\mu_q)$, in particular also in $K$.
Hence, by Lemma \ref{lem:ht-min} and Proposition \ref{prop:lht},
$\hp{p}(K/\Q)\leq \hp{p}(K_q/\Q_q)=s_p(\Q_q)-v_p(p^m)=n-m.$
\end{proof}

If the global field $k$ and a prime $p\neq\charak k$ are fixed 
then a general (i.e.\ random) cyclic extension $K/k$ of degree divisible by $p$
has $\hp{p}(K/k)=0$.
This is the interpretation of

\begin{prop}\label{prop:characters}
Let $k$ be a global field.
Then $X_p(\Gk)/X_p(\Gk)^p$ is infinite for any $p\neq\charak k$.
\end{prop}
(Note that $X_p(\Gk)^p$ is the subgroup consisting of characters of
height greater than zero.)
\begin{proof}
We write $X_p(\Gk)$ additively (only in this proof)
and show that $X_p(\Gk)/pX_p(\Gk)$ is infinite.
Set $t:=s_p(k(\mu_p))$. 
By Chebotarev's density theorem
there are infinitely many primes $\pp$ of $k$ which split completely
in $k(\mu_p)$ and do not split completely in $k(\mu_{p^{t+1}})$. Let
$S=\{\pp_1,...,\pp_n\}$ be any finite set of such $\pp$.  By the
Grunwald-Wang Theorem
\cite[Chapter X, Theorem~5]{artin-tate},
there exists a cyclic extension $L_i/k$ of degree $p^t$ in which
$\pp_i$ is totally ramified, and in which  $\pp_j$  splits
completely for every $j\neq i$.  Let $\chi_i$ be a character of
$\Gk$ corresponding to $L_i$ for $i=1,...,n$.  Let $\chi=\sum
a_i\chi_i$ with some $a_i$, say $a_1$, not divisible by $p$.  Then
if $L$ is the cyclic extension corresponding to $\chi$, the
completion of $L$ at $\pp_1$ coincides with the completion of $L_1$
at $\pp_1$, which does not embed into a cyclic extension of degree
$p^{t+1}$ of $k_{\pp_1}$.  Hence $L/k$ likewise does not embed into
a cyclic extension of degree $p^{t+1}$.  Consequently $\chi\notin
pX_p(\Gk)$. It follows that $\chi_1,...,\chi_n$ are linearly
independent modulo $pX_p(\Gk)$.  Since $n$ can be chosen arbitrarily
large, the proof is complete.
\end{proof}

Since a general cyclic extension $K/k$ does not introduce new roots of unity,
we have actually shown that a general cyclic extension $K/k$ with  $p|[K:k]$
satisfies $b_p(K/k)=s_p(k)$.
This means that the separation of cases (I) and (II) 
is in principle governed only by the number of roots of unity contained in $k$.

\section{Examples}

We determine fibers that contain noncrossed products of degree $8$ and $9$.

\begin{ex}\label{ex:app}
a) Let $k=\Q$.
To answer the question ``For which quadratic number fields $\Q(\chi)$ are there noncrossed products of index $8$ in the fiber over $\chi$?''
we apply Theorem \ref{thm:x} with $m=4$.
By (i) and (ii), these are precisely the quadratic fields $\Q(\chi)$ with
$\hp{2}(\chi)+s_2(\chi)<2$,
or, equivalently, $\hp{2}(\chi)=0$ and $s_2(\chi)=1$.
This condition is satisfied, for instance, by the third cyclotomic field, $\Q(\sqrt{-3})$,
according to Example~\ref{ex:ht}.
Moreover, $\Q(\sqrt{-3})$ is the field of this kind with smallest discriminant
and smallest conductor.
Example \ref{ex:typeii} in Section \ref{sec:density} discusses for which $\alpha\in\Br(\Q)$
the division algebra in $\alpha+\chi$ is a noncrossed product of index $8$
and computes densities.

b) 
Similarly, 
the cubic number fields $\Q(\chi)$ that allow noncrossed products of index $9$ in the fiber over $\chi$
are precisely the ones satisfying
$\hp{3}(\chi)+s_3(\chi)<1$,
or, equivalently, $\hp{3}(\chi)=0$ and $s_3(\chi)=0$.
The smallest such field is the cubic subfield of the $7$-th cyclotomic field
(cf.\ Example \ref{ex:ht} with $p=3$ and $q=7$).

We shall also give two small examples in finite characteristic.

c)
Let $k=\F_3(\bt)$ and let $k(\chi)=\F_3(\bt)(\sqrt{\bt})$.
The prime $\bt$ is tamely ramified in $k(\chi)$
and the residue field of $k$ with respect to $\bt$ has $3$ elements. 
So Proposition \ref{prop:lht} shows $\hp{2}(\chi)\leq 1-1=0$.
Since $s_2(\chi)=1$,
there are noncrossed products of index $8$ in the fiber over $\chi$. 

d)
Let $k=\F_2(\bt)$ and let $k(\chi)$ be generated by a root of $X^3+p(\bt)X+p(\bt)$ where $p(\bt)=\bt^2+\bt+1$. 
(Note that $k(\chi)/k$ is indeed a Galois extension because 
the discriminant of a polynomial of the form $X^3+cX+c$ is always a square in characteristic $2$.)
Since $X^3+p(\bt)X+p(\bt)$ is an Eisenstein polynomial for the prime $p(\bt)$, 
it is irreducible over $k$ and $p(\bt)$ is tamely ramified in $k(\chi)$.
The residue field of $k$ with respect to $p(\bt)$ has $4$ elements, 
so Proposition \ref{prop:lht} shows $\hp{3}(\chi)\leq 1-1=0$.
Evidently, $s_3(\chi)=0$ because $k(\chi)/k$ is cubic.
Thus, there are noncrossed products of index $9$ in the fiber over $\chi$.
\end{ex}

\begin{rem*}%\label{rem:quasi}
Let $k$ be a global field.
If $p\neq\charak k$ and $\mu_p\not\subset k$ then 
there are crossed products
that have a noncrossed product $p$-primary component.
\end{rem*}
\begin{proof}
Let $p$ be as described.
Choose a character $\chi\in X(\Gk)$ of $p$-power order with $s_p(\chi)=0$ and $\hp{p}(\chi)<\infty$.
By Theorem \ref{thm:x},
there are noncrossed products $\alpha+\chi$ of index $|\chi|p^n$, where $n=\hp{p}(\chi)+1$.
(Note that $\chi$ is not exceptional, since $p\neq 2$.)
Let $\chi'\in X(\Gk)$ such that $k(\chi')=k(\mu_p)$.
Since $|\chi'|$ is relatively prime to $|\chi|$, 
we have $\hp{p}(\chi\chi')=\hp{p}(\chi)$ and $s_p(\chi\chi')\geq s_p(\chi')>0$.
Thus, $n\leq \hp{p}(\chi\chi')+s_p(\chi\chi')=b_p(\chi\chi')$.
The product $(\alpha+\chi)\otimes(1+\chi')=\alpha+\chi\chi'$ has index 
$|\chi|p^n\cdot |\chi'|=|\chi\chi'|p^n$,
so it is a crossed product by Theorem~\ref{thm:x}.
\end{proof}

\begin{ex} \label{ex:tensor}
Let $k=\Q, p=3$ and let $k(\chi)$ be the cubic subfield of the $7$-th cyclotomic field
(cf.\ Example \ref{ex:app}b).
Then any noncrossed product $\alpha+\chi$ of index $9$ 
(of which there are infinitely many)
becomes a crossed product when tensored with the quaternion division algebra $(\frac{-3,\bt}{\Q(\!(\bt)\!)})$.
\end{ex}

\section{Brussel's Lemma and Galois covers}
\label{sec:covers}

Let $k$ be a global field.
For any $\Ac\in\Br(k)$ and $\chi\in X(\Gk)$ we write $\ApX$ for the
sum in $\Br(k)\oplus X(\Gk)$, regarded canonically (depending on $\bt$)
as a Brauer class over $\kt$ or $\ktt$. 
This means $\chi$ is identified with the class of the cyclic algebra
$(k(\chi)(\bt)/\kt,\sigma_\chi,\bt)$ resp.\
$(k(\chi)(\!(\bt)\!)/\ktt,\sigma_\chi,\bt)$ defined by $\chi$ and
$\alpha$ is identified with its restriction to $\kt$ resp.\ $\ktt$.
Due to Nakayama \cite{nakayama:witt-thm} we have the index formula
\begin{equation}\label{eq:ind}\ind(\ApX) = |\chi|\cdot\ind \AAA^{k(\chi)},\end{equation}
  where $\AAA^{k(\chi)}$ denotes the restriction of $\AAA$
to the cyclic extension $k(\chi)$. In fact, one commonly proves this
formula first over $\ktt$ and then derives it over $\kt$.

\begin{varthm}[Brussel's Lemma]
  Let $\Ac\in\Br(k)$ and let $\chi\in X(\Gk)$.
  Then $\ApX$ is a crossed product
  if and only if
  there is a Galois extension $M/k$ containing $k(\chi)$ that splits~$\AAA$
  and has degree $[M:k(\chi)]=\ind \AAA^{k(\chi)}$.
\end{varthm}

In the case $\charak k=0$ Brussel's Lemma is \cite{brussel:noncr-prod}, Corollary on p.381, 
but also holds for arbitrary characteristic (cf.\ Hanke \cite{hanke:thesis}, Theorem~5.20).

We start our investigation into the existence of Galois extensions
$M/k$ as in Brussel's Lemma by introducing the following
terminology. Let $K/k$ be a cyclic extension of global fields and
let $m$ be a positive integer. A {\df Galois $m$-cover} of $K/k$ is
an extension field $M\supseteq K$ that is Galois over $k$ and has
degree $m$ over $K$. Since non-Galois covers are not considered in
this paper, we simply use the term {\em cover} and always mean
Galois cover. We call a cover {\em cyclic} if $M/k$ is cyclic. Using
this terminology, $\hp{p}(K/k)$ is the maximal number $r$
such that $K/k$ has a cyclic $p^r$-cover, or infinity if no maximal
$r$ exists.
Moreover, by taking field composita, for any positive integer $m$ with prime factorization $m=\prod_p p^{n_p}$,
\begin{equation}\label{eq:ht-cover}
\textrm{$K/k$ has a cyclic $m$-cover }\iff\textrm{ $n_p\leq\hp{p}(K/k)$ for all $p|m$.}
\end{equation}

Let $\PP$ be a prime of $K$.
An $m$-cover
$M$ of $K/k$ is said to have {\df full local degree at $\PP$} if
$\PP$ is non-archimedian and
$[M_\PP:K_\PP]=m$, or 
if $\PP$ is real and $[M_\PP:K_\PP]=(2,m)$, or
if $\PP$ is complex. Let $S$ be a set of
primes of $K$. An $m$-cover $M$ of $K/k$ is said to have {\em full
local degree in $S$} if it has full local degree at each prime in
$S$. Throughout the paper, $S$ always denotes a finite set of primes
while infinite sets of primes are written $P$.

Before establishing the connection between Brussel's Lemma and
Galois covers with full local degree we recall a few essential facts
about Hasse invariants. For an exposition of this material the
reader is referred to Pierce \cite{pierce:ass-alg}. The Brauer group
of a global field $K$ is given by the exact sequence
\begin{equation}
1 \lra\Br(K)\stackrel{\inv}{\lra} \bigoplus_{\PP\text{ non-archim.}}
\Q/\Z \oplus \bigoplus_{\PP\text{ real}} \frac{1}{2}\Z/\Z
\stackrel{\sum}{\lra} \Q/\Z\lra 1,
  \label{eq:Brk}
\end{equation}
where $\PP$ runs over all primes of $K$. We write $\inv_\PP \AAA$
for the $\PP$-component of $\inv \Ac$, which is called the Hasse
invariant of $\AAA$ at $\PP$.
For any $\Ac\in\Br(K)$, we call the finite set
$$ \supp(\AAA):=\sett{\PP}{\PP \text{ a prime of $K$,} \inv_\PP \AAA\neq 0}$$
the {\em support} of $\AAA$.
(In the literature $\supp(\AAA)$ is sometimes written $\mathrm{ram}(\AAA)$ for ``ramified primes''.)
The index $\ind \AAA_\PP$ of the completion $\AAA_\PP$ is equal to
the order of $\inv_\PP \AAA$, and $\ind \AAA$ is equal to the least
common multiple of all $\ind \AAA_\PP$. We say $\AAA$ has {\em full
local index at $\PP$} if $\PP$ is non-archimedian and $\ind
\AAA_\PP=\ind \AAA$, or if $\PP$ is real and $\ind \AAA_\PP=(2,\ind
\AAA)$, or if $\PP$ is complex.
We define the {\em restricted support} of $\AAA$ by (note the underline)
$$ \SA(\AAA):=\sett{\PP\in \supp(\AAA)}{\text{$\AAA$ has full local index at $\PP$}}.$$
Unless $\alpha=1$, $\SA(\AAA)$ is finite and $\SA(\AAA)\subseteq\supp(\AAA)$.
Under extension of scalars to a finite Galois extension $M\supseteq K$,
\begin{equation}
\inv_{\QQ} \AAA^M=[M_{\PP}:K_\PP]\cdot\inv_\PP \AAA
  \label{eq:inv}
\end{equation}
holds for all primes $\QQ$ of $M$ dividing $\PP$.

\begin{lemma}
  Let $\Ac\in\Br(k)$, $\chi\in X(\Gk)$ and $m=\ind \AAA^{k(\chi)}$.
  \begin{enumerate}
    \item[a)]
  If $\ApX$ is a crossed product then
  there is an $m$-cover of $k(\chi)/k$ with full local degree in $\SA(\AAA^{k(\chi)})$.
\item[b)] If there is an $m$-cover of $k(\chi)/k$ with full local degree in $\supp(\AAA^{k(\chi)})$
  then $\ApX$ is a crossed product.
  \end{enumerate}
  \label{lem:brussel2}
\end{lemma}
\begin{proof}
  If $\ApX$ is a crossed product then, by Brussel's Lemma,
  there is an $m$-cover $M$ of $k(\chi)/k$ that splits $\AAA^{k(\chi)}$.
  Let $\PP\in\SA(\AAA^{k(\chi)})$ and
  let $\QQ$ be a prime of $M$ with $\QQ|\PP$.
  By \eqref{eq:Brk} and \eqref{eq:inv},
  $$0=\inv_\QQ \AAA^M=[M_\PP:k(\chi)_\PP]\inv_\PP \AAA^{k(\chi)}.$$
  Since $\AAA^{k(\chi)}$ has full local index at $\PP$
  it follows that $M$ has full local degree at $\PP$.
  Thus, a) is proved.

  Conversely, any $m$-cover of $k(\chi)/k$ with full local degree in $\supp(\AAA^{k(\chi)})$ splits $\AAA^{k(\chi)}$ by \eqref{eq:Brk} and \eqref{eq:inv},
  hence b) also follows from Brussel's Lemma.
\end{proof}
Quantification over all division algebras of fixed degree in a fiber
yields

\begin{prop}\label{prop:brussel}
  Let $\chi\in X(\Gk)$.
  Then all division algebras in the fiber over $\chi$ of index $|\chi|m$ are crossed products
  if and only if for all finite sets $S$ of primes of $k(\chi)$
  there is an $m$-cover of $k(\chi)/k$ with full local degree in $S$.
\end{prop}
\begin{proof}
  The ``if''-part is immediate from Lemma \ref{lem:brussel2} b).
  The ``only if''-part follows from Lemma \ref{lem:brussel2} a)
  and the fact that for any finite $S$ there exists
  $\Ac\in\Br(k)$ with $\ind \AAA^{k(\chi)}=m$ and $\SA(\AAA^{k(\chi)})\supseteq S$
  (this is a consequence of \eqref{eq:Brk}).
\end{proof}

\begin{rem*}
  By negation of Proposition \ref{prop:brussel},
  the fiber over $\chi$ contains noncrossed products of index $|\chi|m$
  if and only if
  there is a finite set $S$ of primes of $k(\chi)$ such that no $m$-cover of $k(\chi)/k$
  has full local degree in $S$.
\end{rem*}

We use Proposition \ref{prop:brussel} to reformulate Theorem \ref{thm:x} in terms of Galois covers.
Let $K/k$ be a cyclic extension of global fields.
For each prime $p$, define $b_p:=b_p(K/k)$ as in Definition \ref{def:except}.
Let $m$ be a positive integer with prime factorization $m=\prod_p p^{n_p}$.
Then part (i) of Theorem \ref{thm:x} is equivalent to
\begin{theorem} \label{thm1}
  If $n_p\leq b_p$ for all prime factors $p$ of $[K:k]$ then
  for any finite set $S$ of primes of $K$ there is
  an $m$-cover with full local degree in~$S$.
\end{theorem}

Theorem~\ref{thm2} below implies part (ii) of Theorem \ref{thm:x}. 
(This is clear except for the density statement.
It is the purpose of Section \ref{sec:density} below to show that Theorem~\ref{thm2} also implies the density statement.)

\begin{theorem} \label{thm2}
  If $n_p>b_p$ for some prime factor $p$ of $[K:k]$ then
  there are non-archimedian primes $\PP_0,\PP_1,\PP_2$ in $K$
  such that there is no $m$-cover of $K/k$ with full local degree in $\set{\PP_0,\PP_1,\PP_2}$.
  In fact, there are infinite sets $P_0,P_1,P_2$ of non-archimedian primes of $K$
  such that for any $(\PP_0,\PP_1,\PP_2)\in P_0\times P_1\times P_2$
  there is no $m$-cover of $K/k$ with full local degree in $\set{\PP_0,\PP_1,\PP_2}$.
\end{theorem}

\section{Proofs} \label{sec:proof}

We begin with Theorem \ref{thm1}.
We can assume $m$ to be a prime power 
because the compositum of $p^{n_p}$-covers with full local degree in $S$ is clearly an $m$-cover with full local degree in $S$.
So let $m=p^n$ ($p$ prime).
We can further assume that $[K:k]$ is a $p$-power,
because any $p^n$-cover of $K_p/k$ yields a $p^n$-cover of $K/k$.
Under these assumptions Theorem~\ref{thm1} reads:

\begin{theorem}\label{thm1p}
  Let $p$ be a prime and let $[K:k]$ be a $p$-power.
  For any $n\leq b_p$ 
  and any finite set $S$ of primes of $K$ 
  there is a $p^n$-cover of $K/k$ with full local degree in~$S$.
\end{theorem}

\begin{prop}\label{prop1}
  Let $p$ be a prime and let $[K:k]$ be a $p$-power.
  For any $n\leq \hp{p}(K/k)+s_p(K)$ and any finite set $S$ of primes of $K$ 
  there is a $p^n$-cover of $K/k$ with full local degree in~$S$.
\end{prop}
\begin{proof}
  Write $n$ as $n=n'+n''$ with $n'\leq\hp{p}(K/k)$ and $n''\leq s_p(K)$.
  Since $n'\leq\hp{p}(K/k)$, there is a cyclic $p^{n'}$-cover $M'$ of $K/k$.
  We can assume $M'$ to have full local degree in $S$
  (this is a standard argument, see e.g.\ \cite[Proposition 2]{hanke:absolute-galois} for details).
  Now, we invoke the main theorem of \cite{hanke:absolute-galois}:
  since $M'$ contains all $p^{n''}$-th roots of unity
  there is a $p^{n''}$-cover $M$ of $M'/k$ with full local degree in $S$
  (more precisely: in the set of primes of $M'$ dividing a prime in $S$).
  The field $M$ is clearly a $p^n$-cover of $K/k$ with full local degree in $S$.
\end{proof}

\begin{proof}[Proof of Theorem~\ref{thm1p} if $K/k$ is non-exceptional or $p\neq 2$]\mbox{}\\
Since we have $b_p=\hp{p}(K/k)+s_p(K)$ the assertion of Theorem \ref{thm1p} is exactly Proposition \ref{prop1}. 
\end{proof}

The proof of Theorem~\ref{thm1p} if $K/k$ is exceptional and $p=2$ is postponed to Section~\ref{sec:except}.
So far, we have completed the proof of Theorem \ref{thm1} when $K/k$ is non-exceptional or $m$ is odd.

\medskip
We turn to Theorem \ref{thm2},
starting with a reduction of the proof to prime powers $m$.

\begin{lemma}
  Let $m$ be a positive integer.
  If $\charak k=p_0>0$ then write $m=p_0^{n_0}m_0$ with $p_0\ndiv m_0$;
  if $\charak k=0$ then set $m_0=m$.
 
  There are infinitely many primes $\PP_0$ of $K$
  such that any $m$-cover of $K/k$ with full local degree at $\PP_0$
  contains an $m'$-cover of $K/k$ (with full local degree at $\PP_0$) 
  that is abelian over $K$ with $m_0|m'$.
  \label{lem:red}
\end{lemma}
\begin{proof}
  By Chebotarev's density theorem (cf.\ page \pageref{chebot})
  there are infinitely many non-archimedian primes $\PP_0$ of $K$
  that split completely in $K(\mu_{m_0})$.
  Since $\charak k\ndiv m_0$, we can assume $(N(\PP_0),m_0)=1$ for all of them,
  hence $N(\PP_0)\equiv 1\pmod{m_0}$ by \eqref{eq:split1} on page \pageref{eq:split1}.
  Let $\PP_0$ be such a prime and let
  $M$ be an $m$-cover of $K/k$ with full local degree at $\PP_0$.
  Then $N:=\Gal(M/K)\cong\Gal(M_{\PP_0}/K_{\PP_0})$.
  If $\charak k=p_0$ then let $N_0$ be the maximal normal $p_0$-subgroup of $N$;
  if $\charak k=0$ then let $N_0=1$.
  Let $W$ be the wild inertia group of $\PP_0$ in $M/K$.
  If $\charak k=0$ then $W$ is trivial;
  if $\charak k=p_0$ then $W$ is normal in $N$.
  In any case, $W\subseteq N_0$. 
  Let $e_0$ be the tame ramification index of $\PP_0$ in $M/K$.
  Then $e_0|m_0$, so $N(\PP_0)\equiv 1\pmod{m_0}$ implies $N(\PP_0)\equiv 1\pmod{e_0}$.
  Hence, by \eqref{eq:abel-e} on page \pageref{eq:abel-e}, $N/W$ is abelian, 
  so $N/N_0$ is abelian.

  Let $M_0$ be the fixed field of $N_0$.
  Since $N_0$ is a characteristic subgroup of $N$, $N_0$ is normal in $\Gal(M/k)$.
  So $M_0$ is Galois over $k$ and abelian over $K$.
  Clearly, $m_0$ divides $m':=|N:N_0|=[M_0:K]$, so the proof is completed.
\end{proof}

\begin{varthm}[Corollary]
  For any positive integer $m$ 
  there are infinitely many primes $\PP_0$ of $K$
  such that any $m$-cover $M$ of $K/k$ with full local degree at $\PP_0$
  contains a $p^n$-cover $M'$ of $K/k$ (with full local degree at $\PP_0$) for each primary component $p^n$ of $m$ with $p\neq\charak k$.
\end{varthm}
\begin{proof}
  Let $\PP_0$ be a prime as in Lemma \ref{lem:red}
  and suppose $M$ is an $m$-cover of $K/k$ with full local degree at $\PP_0$.
  Let $M_0\subseteq M$ be an $m'$-cover of $K/k$ as in Lemma \ref{lem:red}.
  For any prime $p$, 
  the prime-to-$p$ part $N'$ of the abelian group $\Gal(M_0/K)$ is a characteristic subgroup.
  Since $\Gal(M_0/K)$ is normal in $\Gal(M_0/k)$, also $N'$ is normal in $\Gal(M_0/k)$.
  This proves that $M_0$ contains a $p^n$-cover of $K/k$
  (with full local degree at $\PP_0$), 
  where $p^n$ is the $p$-primary component of $m'$.
  Since the $p$-primary components of $m$ and $m'$ are equal for $p\neq\charak k$,
  we are done.
\end{proof}

Let $P_0$ be the infinite set of primes $\PP_0$ in the Corollary. 
As we will explain, Theorem~\ref{thm2} reduces to 

\begin{theorem}\label{thm:2p}
  Let $p$ be a prime different from $\charak k$.
  For any positive integer $n>b_p$
  there are infinite sets $P_1,P_2$ of primes of $K$
  such that for any $\PP_1\in P_1$ and any $\PP_2\in P_2$
  there is no $p^n$-cover of $K/k$ with full local degree in $\set{\PP_1,\PP_2}$.
\end{theorem}
We show first that Theorem~\ref{thm2} follows from Theorem~\ref{thm:2p}. 
Assume $n_p>b_p$ for some $p$ dividing $[K:k]$.
Then $p\neq\charak k$, for otherwise $b_p=\infty$.
Let $n=n_p$ and let $P_1,P_2$ be as in Theorem~\ref{thm:2p}.
Suppose $M$ is an $m$-cover of $K/k$ with full local degree in
$\set{\PP_0,\PP_1,\PP_2}$ for some $(\PP_0,\PP_1,\PP_2)\in P_0\times
P_1\times P_2$. By choice of~$P_0$, 
the $m$-cover $M$ contains a $p^n$-cover $M'$
of $K/k$. This contradicts Theorem~\ref{thm:2p} because $M'$ has
full local degree in $\set{\PP_1,\PP_2}$. Thus, $M$ cannot exist and
Theorem~\ref{thm2} is proved.

\medskip
We now turn to the proof of Theorem \ref{thm:2p}. For the rest of
this section, the prime $p$ is fixed and different from $\charak k$.
Define the cyclotomic part of $K/k$ to be
$$T:=K\cap k(\mu_{p^\ift}),$$
where $\mu_{p^\ift}$ denotes the group of all $p$-power roots of
unity. Of course, $s_p(T)=s_p(K)$. 

\begin{lemma}
  For any nonnegative integer $n$
  there are infinitely many primes $\PP$ of $K$
  such that any $p^n$-cover of $K/k$ with full local degree at $\PP$
  is abelian over $T=K\cap k(\mu_{p^\ift})$.
  \label{lem:ab}
\end{lemma}
\begin{proof}
  Since any $p^n$-cover of $K/k$ with full local degree at $\PP$
  is also a $p^n$-cover of $K/T$ with full local degree at $\PP$,
  it suffices to assume $T=k$.
  We apply Chebotarev's density theorem to the compositum $K\cdot k(\mu_{p^n})$.
  The extension $K/k$ is cyclic and $K\cap k(\mu_{p^n})=k$,
  so there are infinitely many non-archimedian primes $\pp$ of $k$
  that are inert in $K$ and split completely in $k(\mu_{p^n})$.
  Since $p\neq\charak k$, we can assume $p\ndiv N(\pp)$ for all of them,
  hence $N(\pp)\equiv 1\pmod{p^n}$ by \eqref{eq:split1}.

  Suppose $M$ is a $p^n$-cover of $K/k$ with full local degree at a prime $\PP$
  of $K$ dividing such a $\pp$.
  Since $\PP$ is inert in $K/k$ and $M$ has full local degree at $\PP$,
  we have $\Gal(M/k)\cong\Gal(M_\pp/k_\pp)$.
  Let $e$ be the ramification index of $\pp$ in $M/k$.
  Since $\pp$ is unramified in $K/k$, we have $e|p^n$.
  Thus $N(\pp)\equiv 1\pmod{e}$, and \eqref{eq:abel-e} shows that $M_\pp/k_\pp$
  (hence also $M/k$) is abelian.
\end{proof}

We distinguish two cases.
Say $K/k$ is  {\df Case A} if $k(\mu_{p^{s_p(K)+1}})/k$ is cyclic
and {\df Case B} if $k(\mu_{p^{s_p(K)+1}})/k$ is non-cyclic.

\begin{prop}
  \label{prop:nsp}
  Suppose $K/k$ is Case A.
  There are infinite sets $P_1,P_2$ of primes of $K$ such that
  for any $p^n$-cover $M$ of $K/k$ with $n>s_p(K)$ and
  with full local degree in $\set{\PP_1,\PP_2}$
  for some $\PP_1\in P_1$ and $\PP_2\in P_2$,
  $M$ contains a cyclic $p^{n-s_p(K)}$-cover of $K/k$.
\end{prop}
\begin{proof}
  Let $s=s_p(K)$.
  Let $P_1$ be an infinite set of primes of $K$ as in Lemma \ref{lem:ab}.
  We apply Chebotarev's density theorem to the extensions $K/k$ and $k(\mu_{p^{s+1}})/k$.
  Since they are both cyclic
  there are infinitely many non-archimedian primes $\pp$ of $k$
  that are inert in $K$ as well as in $k(\mu_{p^{s+1}})$.
  Let $P_2$ be an infinite set of primes $\PP$ of $K$
  that divide such a $\pp$.
  Since $p\neq\charak k$, we can assume $p\ndiv N(\PP)$ for all of them.

  Suppose $M$ is a $p^n$-cover of $K/k$ with full local degree in $\set{\PP_1,\PP_2}$
  for some $\PP_1\in P_1$ and $\PP_2\in P_2$.
  Then there is a unique prime $\QQ_2$ of $M$ dividing $\PP_2$.
  Let $I$ be the inertia field and let $p^e$ be the ramification index of $\QQ_2$ over $k$.
  Since $\PP_2$ is inert in $K/k$, the field $I$ is clearly a cyclic $p^{n-e}$-cover of $K/k$.
  It remains to show $e\leq s$.
  Let $\pp_2=\PP_2\cap T$.
  By choice of $\PP_1$, $M/T$ is abelian, hence also $M_{\QQ_2}/T_{\pp_2}$ is abelian.
  Moreover, $M_{\QQ_2}/T_{\pp_2}$ is tame because $p\ndiv N(\PP_2)$.
  Thus, $N(\pp_2)\equiv 1\pmod{p^e}$ by \eqref{eq:abel-e}.
  On the other hand, by \eqref{eq:split1}, $N(\pp_2)\not\equiv 1\pmod{p^{s+1}}$
  because $\pp_2$ is inert in the nontrivial extension $T(\mu_{p^{s+1}})/T$.
  This proves $e\leq s$.
\end{proof}

\begin{lemma}[Case B]\label{lem:sc}
  If  $K/k$ is Case B then
  $k$ is a number field, $p=2$ and $s_2(K)>s_2(k)=1$.
  Moreover,
  \begin{enumerate}
    \item[a)] $\tilde k=\Q(\eta_{2^s})$ where $s:=s_2(K)$,
    \item[b)] $T=k(i)$.
  \end{enumerate}
\end{lemma}
\begin{proof}
Suppose $K/k$ is Case B.
Then $k$ is a number field, $p=2$ and $s_2(K)>1$, 
for otherwise, $k(\mu_{p^{s_2(K)+1}})/k$ would be cyclic.
Let $s=s_2(K)$.
Since $\Gal(k(\mu_{2^{s+1}})/k)\cong\Gal(\tilde k(\mu_{2^{s+1}})/\tilde k)$,
we can conclude from the diagram 
\eqref{eq:diag} on page \pageref{eq:diag}
%in Figure \ref{fig} on page %\pageref{fig} 
that $\tilde k=\Q(\eta_{2^{r}})$ for some $r\leq s$.
But $k(\mu_{2^s})/k$ is cyclic, so $r=s$.
This proves $s(k)=1$ and part a).
Since $i\in\tilde K$ we have $\tilde K=\Q(\mu_{2^s})$.
Hence $T=k\cdot\tilde K=k(i)$.
Thus, b) is proved.
\end{proof}

\begin{prop}\label{prop:sp}
  Suppose $K/k$ is Case B.
  There are infinite sets $P_1,P_2$ of primes of $K$ such that
  any $2^n$-cover of $K/k$ with $n>s_2(K)$ and with full local degree
  in $\set{\PP_1,\PP_2}$ for some $\PP_1\in P_1$ and $\PP_2\in P_2$
  contains a cyclic $2$-cover of $K/k$.
\end{prop}
\begin{proof}
  Let $s=s_2(K)$.
  By Lemma \ref{lem:sc} a), $k$ is a number field, $p=2$, $s>1=s_2(k)$,
  $\eta_{2^s}\in k$ and $\eta_{2^{s+1}}\not\in K$.
  Any non-archimedian prime $\pp$ of $k$ that is inert in $K$
  and does not divide $2$
  has $N(\pp)\equiv -1\pmod{2^s}$.
  Indeed,
  by \eqref{eq:split2}, $\eta_{2^s}\in k$ implies $N(\pp)\equiv \pm 1\pmod{2^s}$,
  and by \eqref{eq:split1}, $\mu_{2^s}\subset K$ and $s>s_2(k)$ imply $N(\pp)\not\equiv 1\pmod{2^s}$.
  The extension $k(\eta_{2^{s+1}})/k$ is quadratic,
  hence $K\cap k(\eta_{2^{s+1}})=k$.
  We apply Chebotarev's density theorem to the compositum $K\cdot k(\eta_{2^{s+1}})$.
  On the one hand,
  there are infinitely many non-archimedian primes $\pp_1$ of $k$ not dividing $2$
  that are inert in $K$ as well as in $k(\eta_{2^{s+1}})$.
  On the other hand,
  there are infinitely many non-archimedian primes $\pp_2$ of $k$ not dividing $2$
  that are inert in $K$ and split completely in $k(\eta_{2^{s+1}})$.
  Owing to \eqref{eq:split2}, we conclude
  $N(\pp_1)\equiv -1+2^s\pmod{2^{s+1}}$ for all $\pp_1$ and
  $N(\pp_2)\equiv -1\pmod{2^{s+1}}$ for all $\pp_2$.
  Hence,
  \begin{equation}\label{eq:equiv}
    N(\pp_1)\not\equiv N(\pp_2)^l \pmod{2^n} \quad\text{for any $n>s$ and any $l\in\N$.}
  \end{equation}
  Let $n>s$ and let $\PP_1,\PP_2$ be the unique primes of $K$ with $\PP_i|\pp_i$.
Assume there is a $2^n$-cover $M$ of $K/k$ with full local degree in
$\set{\PP_1,\PP_2}$ that does not contain a cyclic $2$-cover of
$K/k$. 
Then $\pp_1,\pp_2$ have unique prime divisors in $M$,
so the notation $Z_{\pp_i}(M/k),I_{\pp_i}(M/k),\varphi_{\pp_i}(M/k)$ is unambiguous and $Z_{\pp_i}=k$.
Moreover, $I_{\pp_i}(M/k)\supseteq K$ because $\pp_i$ is inert in $K$.
We conclude $I_{\pp_i}(M/k)=K$, 
for $I_{\pp_i}/Z_{\pp_i}$ is always cyclic and $M$ was assumed not to contain a cyclic $2$-cover of $K/k$. 
Thus, $\varphi_{\pp_1}$ and $\varphi_{\pp_2}$ are both generators of $\Gal(K/k)$.
Let $\varphi_{\pp_1}=\varphi_{\pp_2}^l$ with $l\in\N$.
By \eqref{eq:local-Gal},
$\varphi_{\pp_i}$ acts on $\Gal(F/K)$ as the $N(\pp_i)$-th power map.
This means that $\varphi_{\pp_1}$ acts at the same time as the $N(\pp_1)$-th power map
and as the $N(\pp_2)^l$-th power map on a cyclic group of order $2^n$,
contradicting \eqref{eq:equiv}. 
The proof is thus completed by choosing, for each $i=1,2$,
the set $P_i$ to be the set of primes of $K$ dividing one of the
infinitely many $\pp_i$ as above.
\end{proof}

\begin{proof}[Proof of Theorem~\ref{thm:2p} if $K/k$ is non-exceptional or $p\neq 2$]\mbox{}\\
  Let $K/k$ be non-exceptional or $p\neq 2$ and let  $n>b_p=\hp{p}(K/k)+s_p(K)$.
  Theorem~\ref{thm:2p} claims the existence of infinite sets $P_1,P_2$ of primes of $K$
  such that for any $\PP_1\in P_1$ and any $\PP_2\in P_2$
  there is no $p^n$-cover of $K/k$ with full local degree in $\set{\PP_1,\PP_2}$.
  We choose $P_1,P_2$ depending on the case of $K/k$.

  {\em Case A} :
  Choose $P_1,P_2$ as in Proposition~\ref{prop:nsp}.
  Then, if there exists a $p^n$-cover with full local degree
  in $\set{\PP_1,\PP_2}$ for some $\PP_1\in P_1$ and $\PP_2\in P_2$,
  we can conclude $\hp{p}(K/k)\geq n-s_p(K)$.
  Since this contradicts the hypothesis $n>\hp{p}(K/k)+s_p(K)$,
  there is no $p^n$-cover of $K/k$ with full local degree in $\set{\PP_1,\PP_2}$
  for any $\PP_1\in P_1, \PP_2\in P_2$.

  {\em Case B} : By Lemma \ref{lem:sc}, $k$ is a number field, $p=2$, $-1$ is a square in $K$
and $T=k(i)$.
  Thus, $K/k$ is non-exceptional by asumption, i.e.\ $\hp{2}(K/k)=0$ or $\hp{2}(K/T)=\hp{2}(K/k)$.

  {\em Subcase $\hp{2}(K/k)=0$}:
  Choose $P_1,P_2$ as in Proposition~\ref{prop:sp}.
  Then,
  if there is a $2^n$-cover with full local degree
  in $\set{\PP_1,\PP_2}$ for some $\PP_1\in P_1$ and $\PP_2\in P_2$,
  we can conclude $\hp{2}(K/k)>0$, a contradiction.

  {\em Subcase $\hp{2}(K/T)=\hp{2}(K/k)$}:
  Since every $2^n$-cover of $K/k$ is also a $2^n$-cover of $K/T$,
  it suffices to prove the claim for $K/T$.
  But $K/T$ is Case A because $s_2(T)>1$, hence we are done.
\end{proof}

The proof of Theorem~\ref{thm:2p} if $K/k$ is exceptional and $p=2$ is postponed to Section~\ref{sec:except}.
So far, we have completed the proof of Theorem \ref{thm2} when $K/k$ is non-exceptional or $m$ is odd.
Altogether at this point, Theorem \ref{thm:x} is proved when $K/k$ is non-exceptional or $m$ is odd,
except for the density statement.

\section{Density in the Brauer group}
\label{sec:density}

In this section we define the notion of density in the Brauer group of a global field
and then prove that Theorem \ref{thm2} implies the density statement of Theorem \ref{thm:x},
thus completing the proof of Theorem \ref{thm:x} if $K/k$ is non-exceptional or $m$ is odd.
 
Let $k$ be a global field. In order to measure the ``density'' of
subsets $X\subseteq\Br(k)$ we consider
$$X_S:=\sett{\alpha\in X}{\supp(\alpha)\subseteq S}$$
for growing (finite) sets $S$ of primes of $k$.

\begin{defn}
Suppose $X\subseteq Y\subseteq\Br(k)$ such that $X_S$ and $Y_S$ are
finite for any finite~$S$. Define
$$d_S(X|Y):=\frac{|X_S|}{|Y_S|}.$$
For any $x>0$ let $S_x$ denote the set of non-archimedian primes of
$k$ with absolute norm $\leq x$. We write $d_x(X|Y)$ for
$d_{S_x}(X|Y)$. If the limit $\lim_{x\rightarrow\ift} d_x(X|Y)$
exists we define
$$d(X|Y):=\lim_{x\rightarrow\ift} d_x(X|Y)$$
and call it the {\df natural density} of $X$ in $Y$.
\end{defn}

For finite intersections,
\begin{equation}
  \text{if } d(X_i|Y)=1 \text{ then } d(\bigcap X_i|Y)=1.
  \label{eq:d}
\end{equation}

\begin{lemma}
  \label{lem:d=1}
Let $K/k$ be a given cyclic extension of degree $n$. Let $P$ be any
given infinite set of non-archimedian primes of $K$. For a fixed
integer $m>1$ consider the sets
\begin{align*}
Y&=\sett{\alpha\in\Br(k)}{\ind\alpha^K|m},\\
X&=\sett{\alpha\in\Br(k)}{\ind\alpha^K=m\text{ and
}\SA(\alpha^K)\cap P\neq\emptyset}.
\end{align*}
Then $d(X|Y)=1$.
\end{lemma}
\begin{proof}
  For convenience, denote the infinite set $\sett{\PP\cap k}{\PP\in P}$ also by $P$.
  Fix a non-archimedian prime $\pp_0$ of $k$ that is inert in $K$.
  We can assume without loss of generality $\pp_0\not\in P$
  because replacing $P$ by $P\backslash \{\pp_0\}$ makes $X$ smaller.

For any prime $\pp$ of $k$ write $n_\pp=[K_\pp:k_\pp]$.
Let $S$ be an arbitrary finite set of non-archimedian primes of $k$.
The description of $\Br(k)$ by means of Hasse invariants
(cf.\ \eqref{eq:Brk} and \eqref{eq:inv} on page \pageref{eq:Brk}) 
allows to naturally identify $Y_S$ with
$$ \Sett{y\in\prod_{\pp\in S} \frac{1}{n_\pp m}\Z/\Z}{\sum_{\pp\in S} y_\pp=0}.$$
Because $S$ is growing it is natural to assume that it contains our
fixed prime $\pp_0$. Let $S=S'\dot{\cup}\set{\pp_0}$.
Since $n_{\pp_0}=n$, 
we have an identification
$$ Y_S=\prod_{\pp\in S'} \frac{1}{n_\pp m}\Z/\Z.$$
For any vector $y=(y_\pp)_{\pp\in S'}\in Y_S$ we have
$y_\pp=\inv_\pp y$. If $\ord y_\pp=n_\pp m$ for some $\pp\in S'\cap
P$ then \eqref{eq:inv} implies $\ind y^K=m$, in other words $y\in
X_S$. Hence,
$$ Y_S\backslash X_S\subseteq\SetL{y\in\prod_{\pp\in S'} \frac{1}{n_\pp m}\Z/\Z}{\ord y_\pp<n_\pp m \text{ for all $\pp\in S'\cap P$}}.$$
Counting the vectors in the set on the right yields
$$ \frac{|Y_S\backslash X_S|}{|Y_S|}\leq\prod_{\pp\in S'\cap P} \left(1-\frac{\varphi(n_\pp m)}{n_\pp m}\right)
\leq\prod_{\pp\in S'\cap P} \left(1-\frac{\varphi(nm)}{nm}\right) =
\left(1-\frac{\varphi(nm)}{nm}\right)^{|S'\cap P|}.$$
(Note that if $a|b$, then
$\frac{\varphi(a)}{a}\geq\frac{\varphi(b)}{b}$.)

 Given
$\varepsilon>0$, choose $r$ sufficiently large such that
$(1-\frac{\varphi(nm)}{nm})^r<\varepsilon$ and choose $x$
sufficiently large such that $|S_x'\cap P|\geq r$. We get
$$ d_x(X|Y)=\frac{|X_{S_x}|}{|Y_{S_x}|}\geq 1-\varepsilon.$$
This proves the lemma.
\end{proof}

\begin{theorem}\label{thm:density}
Let $\chi\in X(\Gk)$ and
let $m$ be a positive integer with prime factorization $m=\prod_p p^{n_p}$.
Suppose $\charak k\ndiv m$ and $n_p>b_p$ for some $p$ dividing $|\chi|$.
For the sets
\begin{align*}
B&:=\sett{ \alpha\in\Br(k) }{ \ind(\DaX)=|\chi|m },\\
B_0&:=\sett{ \alpha\in B }{ \DaX \text{ is a noncrossed product} }
\end{align*}
the natural density of $B_0$ in $B$ exists and $d(B_0|B)=1.$
\end{theorem}
Theorem \ref{thm:density} is precisely the density statement of Theorem \ref{thm:x} (ii).

\begin{proof}
The proof depends on Theorem \ref{thm2},
which has been proved so far only for non-exceptional $K/k$ or odd $m$.
However, the proof will show that Theorem \ref{thm:density} holds whenever Theorem \ref{thm2} holds. 

Let $K=k(\chi)$. Let $P_0,P_1,P_2$ be infinite sets of
non-archimedian primes of~$K$ as in Theorem~\ref{thm2}. Define
\begin{align*}
Y&=\sett{\alpha\in\Br(k)}{\ind\alpha^K|m},\\
X_i&=\sett{\alpha\in\Br(k)}{\ind\alpha^K=m\text{ and
}\SA(\alpha^K)\cap P_i\neq\emptyset}
\end{align*}
for $i=0,1,2$. By choice of the $P_i$ and by Lemma
\ref{lem:brussel2} a) we have $X_0\cap X_1\cap X_2\subseteq B_0$,
and $B\subseteq Y$ is obvious from \eqref{eq:ind}. Thus, it suffices
to show $d(X_0\cap X_1\cap X_2|Y)=1$. By Lemma~\ref{lem:d=1},
$d(X_i|Y)=1$ for each~$i$.
The proof is complete because of \eqref{eq:d}.
\end{proof}

This completes the proof of Theorem \ref{thm:x} if $K/k$ is non-exceptional or $m$ is odd.

\begin{ex}\label{ex:typeii}
Let $k=\Q$ and let $K=\Q(\chi)=\Q(\sqrt{-3})$ as in Example \ref{ex:app}a.
Then $s(\chi)=1$ and $\hp{2}(\chi)=0$ (cf.\ Example \ref{ex:ht}).
Thus, Theorem \ref{thm:density} applies with $m=2^n$ for all $n\geq 2$, and we have
\begin{align*}
B&=\sett{ \alpha\in\Br(\Q) }{ \ind(\DaX)=2^{n+1}}=\sett{ \alpha\in\Br(\Q) }{ \ind\alpha^K=2^n },\\
B_0&=\sett{ \alpha\in B }{ \DaX \text{ is a noncrossed product} }.
\end{align*}
Our goal is to describe explicitly a subset of $B_0$ with density $1$ in $B$.
For this purpose,
we simply have to determine the sets $P_1,P_2$ from Proposition \ref{prop:nsp}.
Note that $K/\Q$ is case A.
The odd primes that are inert in $K$ are the primes $\equiv 2\pmod{3}$.
The set $P_1$ consists of the primes that are inert in $K$ 
and split completely in $\Q(\mu_{2^n})$, so
\[
\begin{aligned}
P_1&=\sett{q\in\Ps}{q\equiv 1+2^n&\pmod{3\cdot 2^n}} &&&\textrm{ if $n$ is even,}\\
P_1&=\sett{q\in\Ps}{q\equiv 1+2^{n+1}&\pmod{3\cdot 2^n}} &&&\textrm{ if $n$ is odd.}
\end{aligned}
\]
The set $P_2$ consists of the primes that are inert in $K$ as well as in $\Q(\mu_{2^2})$, so
\begin{align*}
P_2&=\sett{q\in\Ps}{q\equiv -1\pmod{12}}.
\end{align*}
If $\alpha\in B$ satisfies $\ind_{q_i}\alpha^K=4$ at some $q_1\in P_1$ and some $q_2\in P_2$
then $\alpha+\chi$ is a noncrossed product.
This occurs with density $1$ in $B$.
\end{ex}

The remainder of the paper is primarily devoted to the proof of Theorem \ref{thm:x} for exceptional $K/k$ and even $m$.

\section{The Local-Global Principle for the height}\label{sec:height2}

\begin{theorem}\label{thm:lCFT}
Let $K/k$ be a cyclic extension of local fields.
For any prime-power $p^r$,
\[\hp{p}(K/k)\geq r \iff\mu_{p^r}(k)\subset N_{K/k}(K^*).\]
\end{theorem}
(In contrast to Albert's Theorem only those $p^r$-th roots that are actually contained in $k$ are required to be norms here.)
\begin{proof}
This is a consequence of local class field theory and can be found in
Neukirch \cite[Satz (5.1), p.\ 83]{neukirch:einb-zt}.
\end{proof}
\comment{
\begin{proof}
Let $G$ be the pro-$p$-completion of $k^*$, a locally compact abelian group.
Let $\wh G$ denote the character group $\Hom_c(G,T)$, where $T$ is the unit circle.
For any subsets $H\subset G$ and $\Gamma\subset\wh G$ set 
$H^\bot:=\sett{\chi\in\wh G}{\chi(H)=1}$ and
$\Gamma^\bot:=\sett{a\in G}{\chi(a)=1\textrm{ for all $\chi\in\Gamma$}}$.
By Pontrjagin duality, we have $(H^\bot)^\bot=H$ and $(\Gamma^\bot)^\bot=\Gamma$
for closed subgroups $H$ and $\Gamma$.

By local class field theory,
there is a canonical isomorphism of topological groups
between $G$ and the maximal abelian pro-$p$ quotient of $\Gk$.
We will identify these two groups.

Let $K=k(\chi)$,  $\chi\in\wh G$.
The condition $\hp{p}(K/k)\geq r$ is equivalent to $\chi\in{\wh G}^{p^r}$.
Furthermore, by local class field theory,
$N_{K/k}(K^*)=\ker\chi$.
Thus, by definition of $H^\bot$ for $H=\mu_{p^r}$,
the condition $\mu_{p^r}\subset N_{K/k}(K^*)$ is equivalent to $\chi\in\mu_{p^r}^\bot$.
The assertion of the theorem is ${\wh G}^{p^r}=\mu_{p^r}^\bot$.
By Pontrjagin duality, this is equivalent to $({\wh G}^{p^r})^\bot=\mu_{p^r}$,
which we will now prove.
\begin{align*} 
a\in({\wh G}^{p^r})^\bot &\iff \chi(a)=1\text{ for all }\chi\in {\wh G}^{p^r}\\
&\iff \psi^{p^r}(a)=1\text{ for all }\psi\in \wh G\\
&\iff \psi(a^{p^r})=1\text{ for all }\psi\in \wh G \iff a\in\mu_{p^r}.
\end{align*}
\end{proof}
}
\begin{cor}\label{cor:lCFT}
Let $k$ be a local field with $\charak k\neq 2$ and $i\not\in k$.
Then any finite cyclic extension $K/k$ has $\hp{2}(K/k)\in\set{0,\infty}$.
\end{cor}
\begin{proof}
Since $\mu_{2^r}(k)=\set{\pm 1}$, the condition $\mu_{2^r}(k)\subset N_{K/k}(K^*)$ 
is independent of $r$ for all $r\geq 1$.
Thus, by Theorem \ref{thm:lCFT}, $\hp{2}(K/k)>0$ if and only if $\hp{2}(K/k)=\infty$.
\end{proof}

%As an immediate consequence of Theorem \ref{thm:lCFT},
%\begin{equation*}\label{eq:lh3}
%\hp{p}(K/k)\geq s_p(k) \Rightarrow \hp{p}(K/k)=\infty.
%\end{equation*}
%\begin{equation}\label{eq:lh3}
%\hp{p}(K/k)\in [1,s_p(k)-1]\cup\{0,\infty\}.
%\end{equation}

Now, let $K/k$ be a cyclic extension of global fields.
Recall that by Lemma \ref{lem:ht-min},
$\hp{p}(K/k)\leq \hp{p}(K_\pp/k_\pp) \textrm{ for all primes $\pp$ in $k$.}$
\begin{defn}
The global field $k$ is called {\df special with index $s$} 
if $k$ is a number field with $\tilde k=\Q(\eta_{2^s})$ for some $s\geq 2$
and $\wt{k_\pp}\supsetneq\tilde k$ for all primes of $k$.
(cf.\ Geyer-Jensen \cite[Definition 2, p.\ 709]{geyer-jensen})
\end{defn}
Recall that $\wt k$ stands for the intersection $k\cap\Q(\mu_{2^\infty})$.
The condition $\wt{k_\pp}\neq\tilde k$ is in fact only required for even primes,
for it always holds for odd primes.

\begin{theorem}\label{thm:AT}
Let $K/k$ be a cyclic extension of global field and let $p$ be a prime. 
For all $r\in\N$, if $p\neq 2$ or $k$ is not special or $k$ is special with index $\geq r$ then:
\begin{equation}\label{eq:AT1}
\hp{p}(K/k)\geq r \iff \hp{p}(K_\pp/k_\pp)\geq r\textrm{ for all primes $\pp$ in $k$.}
\end{equation}
If $k$ is special with index $s$ then there is an idèle class $c$ of $k$, such that for all $r>s$:
\begin{equation}\label{eq:AT2}
\hp{p}(K/k)\geq r \iff 
\begin{aligned}
&\hp{p}(K_\pp/k_\pp)\geq r\textrm{ for all primes $\pp$ in $k$ and}\\
&\textrm{$c$ is the norm of an idèle class of $K$.}
\end{aligned}
\end{equation}
\end{theorem}
\begin{proof}
Artin-Tate \cite{artin-tate}, Chapter~X, Theorem~6.
(Beware of the misprint:  ``c$_{p^r}$'' should be ``a$_0$'' and ``Lemma 2'' should be ``Theorem 2''.
Furthermore, ``in the special case'' should be understood as ``in the special case with $S_0=\emptyset$ with $r>s$''.)
\end{proof}
 
\begin{cor}\label{cor:except}
If $K/k$ is exceptional then $k$ is special with index $s=\hp{2}(K/k)$.
\end{cor}
\begin{proof}
Let $K/k$ be exceptional.
By Definition \ref{def:except}, 
$k$ is a number field, $i\in K$ and $\hp{2}(K/k(i))>\hp{2}(K/k)>0$. 
We prove the statement by contradiction,
so assume $k$ is not special or special with index $s\neq\hp{2}(K/k)$.
If $k$ is not special or special with $s>\hp{2}(K/k)$ then \eqref{eq:AT1} holds for all $r\in\N$.
If $k$ is special with $s<\hp{2}(K/k)$ then \eqref{eq:AT2} applied with $r=\hp{2}(K/k)$ shows that
$c$ is a norm.
Since $c$ is independent of $r$, \eqref{eq:AT1} holds for $r:=\hp{2}(K/k)+1$ in both cases.
Hence, there is a prime $\pp$ in $k$ such that 
$\hp{2}(K_\pp/k_\pp)<r$.
Since $0<\hp{2}(K/k)\leq\hp{2}(K_\pp/k_\pp)<\infty$,
we conclude from Corollary \ref{cor:lCFT} that $i\in k_\pp$.
Therefore,
$\hp{2}(K/k(i))\leq\hp{2}(K_\pp/k(i)_\pp)=\hp{2}(K_\pp/k_\pp)<r$. 
This means $\hp{2}(K/k(i))\leq r-1=\hp{2}(K/k)$ in contradiction to $K/k$ being exceptional,
so the claim is proved.
\end{proof}

\begin{rem*}
Let $k$ be special with index $s$.
Then any finite cyclic extension $K/k$ with $i\in K$ is Case B (for $p=2$).
In particular, $s=s_2(K)$.  
\end{rem*}
\begin{proof}
If $i\in K$ then $s_2(K)+1\geq 3$.
Since $\tilde k=\Q(\eta_{2^s})$, the extension $k(\mu_{2^{s_2(K)+1}})/k$ is non-cyclic,
i.e.\ $K/k$ is Case B.
By Lemma \ref{lem:sc}, $s=s_2(K)$.
\end{proof}

\section{Embedding Problems and Galois covers}

Let $K/k$ be a Galois extension of fields and let $G=\Gal(K/k)$.
Suppose an embedding of $K$ into the separable closure of $k$ is fixed,
and let $\varphi:\Gk\to G$ be the canonical surjection
where $\Gk$ denotes the absolute Galois group of $k$.
An {\em embedding problem} for $K/k$ is a diagram
\begin{equation}\label{eq:ep}
  \begin{diagram}
    \node[4]{\Gk}\arrow{s,b}{\varphi}\arrow{sw,t,..}{\psi}\\
    \node{1}\arrow{e}\node{A}\arrow{e}\node{E}\arrow{e,b}{\pi}\node{G}\arrow{e}\node{1}
  \end{diagram}
\end{equation}
where the bottom row is a group extension.
By a {\em solution} to the embedding problem we mean either a homomorphism $\psi$ that makes the diagram commute
or the fixed field of the kernel of such $\psi$.
We speak of a {\em proper solution} if $\psi$ is surjective,
or $A\subset\img\psi$.
Another embedding problem 
\begin{equation}\label{eq:ep-dom}
  \begin{diagram}
    \node[4]{\Gk}\arrow{s,b}{\varphi}\arrow{sw,t,..}{\psi}\\
    \node{1}\arrow{e}\node{B}\arrow{e}\node{\Gamma}\arrow{e,b}{\sigma}\node{G}\arrow{e}\node{1}
  \end{diagram}
\end{equation}
is said to {\em dominate} \eqref{eq:ep} if there is a commuting diagram 
\[\begin{CD}
1@>>> B @>>> \Gamma @>\sigma>> G @>>> 1\\
@. @VVV @VVV @|\\
1@>>>  A @>>>  E @>\pi>> G @>>> 1
\end{CD} \]
with surjective vertical arrows.
In this case any proper solution to \eqref{eq:ep-dom} yields a proper solution to \eqref{eq:ep}.

Now let $k$ be a global field.
For each prime $\pp$ in $k$, write $k_\pp$ for the completion and
$\Gp$ for the absolute Galois group of $k_\pp$.
We regard $\Gp$ as a subgroup of $\Gk$.
Let $K/k$ be a finite Galois extension, $G=\Gal(K/k)$,
and let $\varphi:\Gk\to G$ be the canonical surjection.
We set $G_\pp:=\varphi(\Gp)$ and define $\varphi_\pp:\Gp\to G_\pp$ as a restriction of $\varphi$.
In this way, we have associated to \eqref{eq:ep}
for each prime $\pp$ in $k$ a corresponding local diagram
\begin{equation}\label{eq:epp}
  \begin{diagram}
    \node[4]{\G{\pp}}\arrow{s,b}{\varphi_\pp}\arrow{sw,t,..}{\psi_\pp}\\
    \node{1}\arrow{e}\node{A}\arrow{e}\node{E_\pp}\arrow{e,b}{\pi}\node{G_\pp}\arrow{e}\node{1}
  \end{diagram}
\end{equation}
where $E_\pp=\pi^{-1}(G_\pp)\subseteq E$.
For each prime $\PP$ in $K$ dividing $\pp$ we can identify $\Gal(K_\PP/k_\pp)$
with $G_\pp$.
(In fact, $G_\pp$ is one of the decomposition groups for $\pp$ in $\Gal(K/k)$.)
Thus, \eqref{eq:epp} is regarded as an embedding problem for $K_\PP/k_\pp$.
Any global solution $\psi$ gives local solutions $\psi_\pp$ at all $\pp$
by restricting $\psi$ to $\G{\pp}$,
but they are not necessarily proper even if $\psi$ is proper.

An $m$-cover of $K/k$ is by definition the proper solution to an embedding problem of the form \eqref{eq:ep} with $|A|=m$ and arbitrary $E$. 
For a non-archimedian prime $\PP$ in $K$, 
the $m$-cover given by $\psi$ has full local degree at $\PP$ if and only if the corresponding local solution $\psi_\PP$ is also proper,
i.e.\ if and only if $A\subset\img(\psi_\PP)$.
For a real prime $\PP$ in $K$, 
the $m$-cover given by $\psi$ has full local degree at $\PP$ if and only if the corresponding local solution $\psi_\PP$ satisfies $|\img(\psi_\PP)|=(2,m)$.

A related well-studied topic is {\em embedding problems with local prescription},
which exactly prescribe local solutions $\psi_\pp$ at finitely many primes rather than just local degrees.
Even though this problem is more restrictive than ours,
we are required to make use of this theory 
in order to deal with exceptional case.
The crucial result for this purpose is Theorem \ref{thm:ep} below.

We regard $A$ as a $\Gk$-module through $\varphi$.
Let $A'=\Hom(A,\mu)$ denote the $\Gk$-module dual to $A$,
where $\mu$ is the natural $\Gk$-module of all roots of unity over $k$.
We write $k(A')$ for the fixed field of all $\sigma\in\Gk$ that fix $A'$ pointwise,
and we set $G':=\Gal(k(A')/k)$.
Note that
\begin{equation}\label{eq:Amu} 
k(A')\subseteq K(\mu_l),\quad\textrm{where $l=\exp A$.}
\end{equation}

\begin{theorem}\label{thm:ep}
Let $A$ be abelian.
Suppose that for all subgroups $U\leq G'$ the map
\begin{equation}\label{eq:H1U}
 H^1(U,A')\stackrel{\res}{\to}\prod_{\sigma\in U} H^1(\gen{\sigma},A')
\end{equation}
is injective.
If the embedding problem \eqref{eq:ep} has local solutions everywhere
then it has a global solution with any local prescription at finitely many primes.
\end{theorem}
\begin{proof}
We first recall two facts about embedding problems with abelian kernel from Neukirch \cite{neukirch:einb-zt}.
If the map
\begin{equation}\label{eq:H2}
 H^2(\Gk,A)\stackrel{\res}{\to}\prod_\pp H^2(\Gp,A)
\end{equation}
is injective then we have a local-global principle saying: 
there is a global solution if and only if there are local solutions everywhere
\cite[(2.2)]{neukirch:einb-zt}.
If there is a global solution and the map 
\begin{equation}\label{eq:H1}
 H^1(\Gk,A)\stackrel{\res}{\to}\prod_{\pp\in S} H^1(\Gp,A)
\end{equation}
is surjective then any local prescription at $S$ can be realized \cite[(2.5)]{neukirch:einb-zt}.

Next, we pass to the dual module.
For each prime $\pp$ in $k$ choose in $G'$ a decomposition group $G'_\pp$.
The  map \eqref{eq:H2} is injective if and only if
\begin{equation}\label{eq:H2dual}
 H^1(G',A')\stackrel{\res}{\to}\prod_\pp H^1(G'_\pp,A')
\end{equation}
is injective \cite[(4.5)]{neukirch:einb-zt}.
Furthermore, the map \eqref{eq:H1} is surjective if 
\begin{equation}\label{eq:H1dual}
 H^1(G'_\pp,A')\stackrel{\res}{\to}\prod_{\sigma\in G'_\pp} H^1(\gen{\sigma},A')
\end{equation}
is injective for all $\pp\in S$ \cite[(6.4a)]{neukirch:einb-zt}.

Now, \eqref{eq:H1dual} is injective for each $\pp$ by hypothesis.
Due to Chebotarev's density theorem, every $\sigma\in G'$ lies in $G'_\pp$ for some $\pp$.
So the injectivity of \eqref{eq:H2dual} follows from the hypothesis for $U=G'$.  
This completes the proof.
\end{proof}

Theorem \ref{thm:ep} holds trivially if $G'$ is cyclic.
For the easiest non-trivial case we provide 

\begin{lemma}\label{lem:H1}
For any bicyclic group $G=\gen{\sigma}\gprod\gen{\tau}$ and  $G$-module $M$ (multiplicatively written), the kernel of the map
\begin{equation*}
H^1(G,M) \lra H^1(\gen{\sigma},M)\times H^1(\gen{\tau},M)\times H^1(\gen{\sigma\tau},M)
\end{equation*}
is isomorphic to the group $Q$ defined as follows:
For any $H\leq G$ let $M^H$ denote the $H$-invariants.
Consider the norm map $N_\sigma: M^{\gen{\tau}}\lra M^G$.
Then 
$$ Q := (\ker(N_\sigma)\cap M^{\sigma-1}\cap M^{\sigma\tau-1})/
(M^{\gen{\tau}})^{\sigma-1}. $$
\end{lemma}
\begin{proof}
Let $l:G\lra M$ be a 1-cocycle. We assume that $l$ is normalized ($l_1=1$). 
Of course, $l$ is fully determined by $l_\sigma, l_\tau$ and the cocycle conditions.
If $l$ is in the kernel then we can further assume w.l.o.g.\ that $l_{\tau}=1$.
Let $x=l_\sigma\in M$, the only remaining parameter. 
The cocycle conditions imply $x=l_{\sigma\tau}=l_{\tau\sigma}=x^\tau$, so $x \in M^{\gen{\tau}}$, as well as $N_\sigma(x)=1$. 
Since the restrictions of $l$ to $\gen{\sigma}$ and $\gen{\sigma\tau}$ also split,
we have $x\in M^{\sigma-1}$ and $x\in M^{\sigma\tau-1}$.
Conversely, any such $x$ actually defines a 1-cocycle (with $l_\tau=1$) because $N_\sigma(x)=1$.
 
Finally, $l$ is split if and only if there exists $a\in M$ with $a^{\tau-1}=l_\tau=1$, i.e.\ $a\in M^{\gen{\tau}}$, and $a^{\sigma-1}=l_\sigma=x$. 
This is the assertion.
\end{proof}

We point out that in the setup of the following section, $K/k$ is always cyclic.
By \eqref{eq:Amu}, $k(A')\subseteq K(\mu_l)$, $l=\exp A$. 
Hence, $G'$ is cyclic or bicyclic unless $\charak k=0, 8|l, \tilde k$ is real and $i\not\in K$.
As we will not encounter this case, Lemma \ref{lem:H1} is sufficient for our purposes.

\section{Metacyclic $2$-groups}\label{sec:metacyclic}

A notable exception in the classification of metacyclic $p$-groups is 
the fact that the following two presentations give isomorphic groups
(cf.\ \cite[Theorem 22]{liedahl:metacyclic-pres}): 
\begin{align*}
&\gen{a,c \,|\, a^{2^{s+1}}=1, c^{2^{t}}=a^{2^s}, a^c=a^{-1}}\\
\cong&\gen{a,c \:|\: a^{2^{s+1}}=1, c^{2^{t}}=a^{2^s}, a^c=a^{-1+2^s}},
\end{align*}
for all $s,t\geq 2$.
Indeed, an isomorphism from left to right is established by mapping $a$ to $ac^{2^{t-1}}$.

Let $1\leq l<t$ and let $C_l$ be a cyclic group of order $2^l$ with generator $\rho$.
Then the isomorphism above is even an isomorphism between group extensions
$$
\minCDarrowwidth12pt
\begin{CD} 
1@>>>\gen{a,c^{2^l}}@>>>\gen{a,c \,|\, a^{2^{s+1}}=1, c^{2^{t}}=a^{2^s}, a^c=a^{-1}}@>>> C_l@>>> 1\\
@. @VVV @VVV @| \\
1@>>>\gen{a,c^{2^l}}@>>>\gen{a,c \,|\, a^{2^{s+1}}=1, c^{2^{t}}=a^{2^s}, a^c=a^{-1+2^s}}@>>> C_l@>>> 1,
\end{CD}
$$
where $c$ maps to $\rho$ in both rows.
We denote this extension by 
\begin{equation}\label{eq:E}
 1\to A_l\to E_{s,t}\to C_l\to 1.
\end{equation}
and choose for $E_{s,t}$ in the sequel the presentation
$$E_{s,t}=\gen{a,c \:|\: a^{2^{s+1}}=1, c^{2^{t}}=a^{2^s}, a^c=a^{-1}}.$$
Clearly, $|E_{s,t}|=2^{s+t+1}$.
The center, commutator subgroup and socle are $\gen{c^2}, \gen{a^2}$ and $\gen{a^{2^s}}$ respectively.
For each $1\leq l<t$, the kernel $A_l$ is abelian and decomposes as follows into cyclic factors.
{Case $t-l\leq s$:} $A_l=\gen{a}\times\gen{a^{2^{l-(t-s)}}c^{2^l}}$, $\exp A_l=2^{s+1}$.
{Case $t-l\geq s$:} $A_l=\gen{c^{2^l}}\times\gen{ac^{2^{t-s}}}$, $\exp A_l=2^{t+1-l}$.

We now turn to investigate the embedding problem defined by \eqref{eq:E}.
Let $K/k$ be any cyclic extension of order $2^l (l\geq 1)$ with $\Gal(K/k)=\gen{\rho}$.
We write $C_l$ for $\Gal(K/k)$ and consider the embedding problem
\begin{equation}\label{eq:E2}
  \begin{diagram}
    \node[4]{\Gk}\arrow{s,b}{\varphi}\arrow{sw,t,..}{\psi}\\
    \node{1}\arrow{e}\node{A_l}\arrow{e}\node{E_{s,t}}\arrow{e}\node{C_l}\arrow{e}\node{1}
  \end{diagram}
\end{equation}

We first point out that \eqref{eq:E2} is dominated in two ways by a split embedding problem.
Indeed, defining for all $s,t\geq 2$ the groups 
\begin{align*}
\Gamma_{s,t}&=\gen{a,c \:|\: a^{2^{s+1}}=1, c^{2^{t}}=1, a^c=a^{-1}},\\
\Delta_{s,t}&=\gen{a,c \:|\: a^{2^{s+1}}=1, c^{2^{t}}=1, a^c=a^{-1+2^s}},
\end{align*}
one obtains a dominating embedding problem for \eqref{eq:E2} by replacing the group extension by either
\begin{equation}\label{eq:Gam}
 1\to \gen{a,c^{2^l}}\to \Gamma_{s,t+1}\to C_l\to 1,
\end{equation}
or
\begin{equation}\label{eq:Del}
 1\to \gen{a,c^{2^l}}\to \Delta_{s,t+1}\to C_l\to 1.
\end{equation}
(Here, $c$ maps to $\rho$ in all three group extensions.)

\begin{lemma}\label{lem:E2sol}
Let $k$ be any field with $\charak k=0$ and $i\not\in k$.
Let $\Gal(K/k)=C_l$ with $l\geq 1$ and let $i\in K$.
If $k$ has more than one quadratic extension
then \eqref{eq:E2} has a proper solution
for all $s<s_2(k(i))$ and all $t<\hp{2}(K/k)+l$.
\end{lemma}
\begin{proof}
Note that $\Gamma_{s,t+1},\Delta_{s,t+1}$ are semi-direct products $\gen{a}\rtimes\gen{c}$ respectively with $\ord a=2^{s+1}$ and $\ord c=2^{t+1}$.
Let $h=t+1-l$. Since $h\leq\hp{2}(K/k)$ there is a cyclic $2^h$-cover $L$ of $K/k$.
We have $[L:k]=2^{t+1}$.
Choose any $a\in k$ with $\sqrt{a}\not\in k(i)$.
Then $M:=L(\sqrt[2^{s+1}]{a})$ is a solution field for either \eqref{eq:Gam} or \eqref{eq:Del}.
Indeed, since $\mu_{2^{s+1}}\in k(i)\subseteq L$, the extension $M/k$ is Galois and $M/L$ is a Kummer extension.
The action of $\Gal(L/k)=\gen{c}$ on $\Gal(M/L)=\gen{a}$ is given by the action of $\Gal(L/k)$ on $\mu_{2^{s+1}}$.
Since $\mu_{2^{s+1}}\subset k(i)$, this action is of order~$2$.
The possible actions of this order are $\zeta\mapsto\zeta^q$ for $q=-1,-1+2^s$ or $1+2^s$.
Since $i$ is not fixed, $q$ is either $-1$ or $-1+2^s$. 
We have thus shown that $M$ is a proper solution for one of the dominating embedding problems.
\end{proof}

\begin{lemma}\label{lem:E2}
Let $k$ be a number field with $\tilde k=\Q(\eta_{2^s}), s\geq 2$.
Let $\Gal(K/k)=C_l$ with $l\geq 1$ and let $i\in K$.
For any $t>l$,
Theorem~\ref{thm:ep} applies to the embedding problem \eqref{eq:E2}.
\end{lemma}
\begin{proof}
We give the proof only for $t-l\leq s$.
The case $t-l\geq s$ is similar.
Set $h:=t-l>0$.
(The example will be used only once, in the proof of Proposition \ref{prop1b},
and this use is for $t-s=l$, i.e.\ $h=s$.)
For $t\leq s+l$ we have 
\begin{align*}
E_{s,t}&=\gen{a,c \:|\: a^{2^{s+1}}=1, c^{2^{l+h}}=a^{2^s}, a^c=a^{-1}},\\
A_l&=\gen{a,c^{2^l}}=\gen{a}\gprod\gen{b}\textrm{ with $b:=a^{2^{s-h}}c^{2^l}$},
\end{align*}
$\exp A_l=\ord a=2^{s+1}, \ord b=2^h$.
Note $b^c=a^{-2^{s-h+1}}b$.
The action of $\Gal(K/k(i))=\gen{\rho}$ on $A$ is trivial because $c^2$ lies in the center of $E_{s,t}$.
Thus, $k(A')\subseteq k(\mu_{2^{s+1}})$.

We continue to determine $k(A')$ and $G'$ exactly.
Let $\zeta=\zeta_{2^{s+1}}$.
Then $A'=\gen{a^*}\gprod\gen{b^*}$ where
\begin{align*}
a^* : a\mapsto \zeta,\quad &b\mapsto 1,\\
b^* : a\mapsto 1, \quad &b\mapsto \zeta^{2^{s+1-h}}.
\end{align*}
The action of $\Gk$ on $A'$ is through $G_0:=\Gal(k(\mu_{2^{s+1}})/k)$.
This group is the Klein 4-group, 
generated for instance by $\sigma,\tau$ with $\sigma:\zeta\mapsto\zeta^{-1}$ and
$\tau:\zeta\mapsto\zeta^{2^s+1}$.
Since $\sigma$ restricts to a generator of $\Gal(k(i)/k)$,
it acts on $A$ like $\rho$ does.
Hence:
\begin{align*}
\sigma(a^*)(a)&=\sigma(a^*(a^c))=\sigma(a^*(a^{-1}))=\sigma(\zeta^{-1})=\zeta,\\
\sigma(a^*)(b)&=\sigma(a^*(b^c))=\sigma(a^*(a^{-2^{s+1-h}}b))=\sigma(\zeta^{-2^{s+1-h}})=\zeta^{2^{s+1-h}},\\
\sigma(b^*)(a)&=\sigma(b^*(a^c))=\sigma(b^*(a^{-1}))=1,\\
\sigma(b^*)(b)&=\sigma(b^*(b^c))=\sigma(b^*(a^{-2^{s+1-h}}b))=\sigma(\zeta^{2^{s+1-h}})=\zeta^{-2^{s+1-h}},
\end{align*}
i.e.\ $\sigma(a^*)=a^*b^*, \sigma(b^*)={b^*}^{-1}.$
Since $\tau$ restricts to the identity on $k(i)$,
it acts trivially on $A$.
Hence:
\begin{align*}
\tau(a^*)(a)&=\tau(a^*(a))=\tau(\zeta)=\zeta^{2^s+1},\\
\tau(a^*)(b)&=\tau(a^*(b))=1,\\
\tau(b^*)(a)&=\tau(b^*(a))=1,\\
\tau(b^*)(b)&=\tau(b^*(b))=\tau(\zeta^{2^{s+1-h}})=\zeta^{2^{s+1-h}},
\end{align*}
i.e.\ $\tau(a^*)=(a^*)^{2^s+1}, \tau(b^*)=b^*.$\\
Furthermore: $\sigma\tau(a^*)=(a^*)^{2^s+1}b^*, \sigma\tau(b^*)={b^*}^{-1}.$
Neither $\sigma$ nor $\tau$ act trivially on $A$,
hence $k(A')=k(\mu_{2^{s+1}})$ and $G'=G_0=\gen{\sigma,\tau}$.

It suffices to check the injectivity of \eqref{eq:H1U} for $U=G'$,
the only non-cyclic subgroup of $G'$.
Easy calculations yield $A'^{\sigma-1}=\gen{b^*}$ and 
$A'^{\sigma\tau-1}=\gen{{a^*}^{2^s}b^*}$,
hence $A'^{\sigma-1}\cap A'^{\sigma\tau-1}=\gen{{b^*}^2}$.
On the other hand, $b^*\in A'^{\gen{\tau}}$ and $(b^*)^{\sigma-1}={b^*}^{-2}$.
This shows that $Q$ vanishes in Lemma~\ref{lem:H1} with $A'$ for $M$ and $G'$ for $G$.
%$A'^{\gen{\tau}}=\gen{b^*,{a^*}^2}$ and 
%$(A'^{\gen{\tau}})^{\sigma-1}=\gen{{b^*}^2}$. 
\end{proof}

\section{Proofs (exceptional case)}\label{sec:except}

This section completes the proof of Theorem \ref{thm:x} for exceptional $K/k$ and even $m$.
It remains to prove Theorems \ref{thm1p} and \ref{thm:2p} for $K/k$ exceptional and $p=2$.
So let $K/k$ be a cyclic extension of global fields that is exceptional, i.e.\ 
$k$ is a number field, $i\in K$ and $\hp{2}(K/k(i))>\hp{2}(K/k)>0$. 
We begin with Theorem \ref{thm1p}.

\begin{prop}\label{prop1b}
  Supppose $K/k$ is exceptional and $[K:k]$ is a $2$-power.
  Let $n=\hp{2}(K/k)+s_2(K)+1$.
  For any finite set $S$ of primes of $K$ 
  there is a $2^n$-cover with full local degree in~$S$.
\end{prop}
\begin{proof}
Let $\Gal(K/k)=C_l=\gen{\rho}$.
By Corollary~\ref{cor:except} on page \pageref{cor:except} and the Remark following it, 
$k$ is special with index  $s=s_2(K)$.
By Lemma \ref{lem:E2}, Theorem \ref{thm:ep} applies to the embedding problem \eqref{eq:E2} for all $t>l$.
Setting $t:=l+\hp{2}(K/k)$, any solution to \eqref{eq:E2} is a $2^n$-cover of $K/k$
because $|A_l|=2^{s+t+1-l}=2^{s+\hp{2}(K/k)+1}=2^n$.
Using Theorem \ref{thm:ep}, it remains to prove the existence of proper local solutions at all non-archimedian primes $\pp$ in $k$. 
(All archimedian primes become complex in $K$ since $i\in K$.)
So let $\pp$ be a non-archimedian prime in $k$ and let $G_\pp=\gen{\rho^{2^g}}$, $0\leq g\leq l$.

{\bf Case 1:} $g=0$. 
In this case, $G_\pp=C_l$ and $E_\pp=E_{s,t}$,
so the local embedding problem is identical to \eqref{eq:E2} with $K/k$ replaced by $K_\pp/k_\pp$.
According to Lemma \ref{lem:E2sol},
it remains to verify $s<s_2(k_\pp(i))$ and $t<\hp{2}(K_\pp/k_\pp)+l$.

Since $g=0$, $\pp$ is not split in $k(i)$, i.e.\ $k_\pp(i)/k_\pp$ is quadratic.
On the other hand, $k_\pp(\mu_{2^{s+1}})/k_\pp$ is at most quadratic,
for $k$ is special with index $s$. 
Thus, $\mu_{2^{s+1}}\subset k_\pp(i)$, i.e.\ $s_2(k_\pp(i))>s$.
By Lemma \ref{lem:ht-min}, we have $\hp{2}(K_\pp/k_\pp)\geq\hp{2}(K/k)>0$.
Since $i\not\in k_\pp$, Corollary \ref{cor:lCFT} implies $\hp{2}(K_\pp/k_\pp)=\infty$.

{\bf Case 2:} $g>0$. 
In this case, $E_\pp=\gen{a,c^{2^g}}$ is abelian, 
for $c^2$ lies in the center of $E_{s,t}$.
By Corollary \ref{cor:except} we have $s=\hp{2}(K/k)$, i.e.\ $t=l+s$.
Thus, $\ord c^{2^g}=2^{s+1+(l-g)}\geq 2^{s+1}=\ord a$, so
$ E_\pp=\gen{c^{2^g}}\times\gen{ac^{2^l}}\cong C_{2^{s+1+(l-g)}}\times C_{2^s}.$
A solution to
\divide\dgARROWLENGTH by2
\begin{equation}
  \begin{diagram}
    \node[4]{\G{\pp}}\arrow{s,b}{\varphi_\pp}\arrow{sw,t,..}{\psi_\pp}\\
    \node{1}\arrow{e}\node{\gen{a,c^{2^l}}}\arrow{e}\node{\gen{c^{2^g}}\times\gen{ac^{2^l}}}\arrow{e}\node{\gen{\rho^{2^g}}}\arrow{e}\node{1}
  \end{diagram}
\end{equation}

\bigskip
is then obtained by composing
a cyclic $2^{s+1}$-cover of $K_\pp/k_\pp$ with 
a Kummer extension of $k_\pp$ of degree $2^s$.
(Note that $2^{s+1}$ is the order of $c^{2^l}$.)
It remains to verify $\hp{2}(K_\pp/k_\pp)\geq s+1$ and $s_2(k_\pp)\geq s$.

Since $g>0$, $\pp$ splits in $k(i)$, i.e.\ $i\in k_\pp$. 
By Lemma \ref{lem:ht-min} on page \pageref{lem:ht-min} and the fact that $K/k$ is exceptional we have
$\hp{2}(K_\pp/k_\pp)=\hp{2}(K_\pp/k_\pp(i))\geq\hp{2}(K/k(i))>\hp{2}(K/k)=s$.
Finally, $\eta_{2^s},i\in k_\pp$ imply $s_2(k_\pp)\geq s$.
\end{proof}

\begin{proof}[Proof of Theorem~\ref{thm1p} for $K/k$ exceptional and $p=2$]\mbox{}\\
Since we have $n\leq b_2=\hp{2}(K/k)+s_2(K)+1$,
the assertion of Theorem~\ref{thm1p} follows from Propositions \ref{prop1} and \ref{prop1b} together.
\end{proof}

We now turn to Theorem \ref{thm:2p}.

\begin{prop}
  \label{prop:nspB}
  Suppose $K/k$ is exceptional.
  There are infinite sets $P_1,P_2$ of primes of $K$ such that
  for any $2^n$-cover $M$ of $K/k$ with $n>s_2(K)$ and
  with full local degree in $\set{\PP_1,\PP_2}$
  for some $\PP_1\in P_1$ and $\PP_2\in P_2$,
$M$ contains a cyclic $2^{n-s_2(K)-1}$-cover of $K/k$.
\end{prop}
\begin{proof}
  The proof is a modification of the proof of Proposition \ref{prop:nsp}.
  Let $s=s_2(K)$.
  By Proposition \ref{prop:nsp}, we can assume $K/k$ is Case~B,
  since any cyclic $2^{n-s}$-cover contains a cyclic $2^{n-s-1}$-cover of $K/k$.

  Let $P_1$ be an infinite set of primes of $K$ as in Lemma \ref{lem:ab}.
  We apply Chebotarev's density theorem to the extensions $K/k$ and $k(\eta_{2^{s+1}})/k$.
  Since they are both cyclic ($k(\eta_{2^{s+1}})/k$ is quadratic by Lemma \ref{lem:sc}),
  there are infinitely many non-archimedian primes $\pp$ of $k$
  that are inert in $K$ as well as in $k(\eta_{2^{s+1}})$.
  Let $P_2$ be an infinite set of primes $\PP$ of $K$
  that divide such a $\pp$.
  Since $\charak k\neq 2$ ($k$ is a number field), we can assume $N(\PP)$ is odd for all of them.

  Suppose $M$ is a $2^n$-cover of $K/k$ with full local degree in $\set{\PP_1,\PP_2}$
  for some $\PP_1\in P_1$ and $\PP_2\in P_2$.
  Then there is a unique prime $\QQ_2$ of $M$ dividing $\PP_2$.
  Let $I$ be the inertia field and let $2^e$ be the ramification index of $\QQ_2$ over $k$.
  Since $\PP_2$ is inert in $K/k$, the field $I$ is clearly a cyclic $2^{n-e}$-cover of $K/k$.
  It remains to show $e\leq s+1$.
  Let $\pp_2=\PP_2\cap T$.
  By choice of $\PP_1$, $M/T$ is abelian, hence also $M_{\QQ_2}/T_{\pp_2}$ is abelian.
  Moreover, $M_{\QQ_2}/T_{\pp_2}$ is tame because $N(\PP_2)$ is odd.
  Thus, $N(\pp_2)\equiv 1\pmod{2^e}$ by \eqref{eq:abel-e} on page \pageref{eq:abel-e}.

  On the other hand, since $\eta_{2^{s}}\in k$ (Lemma \ref{lem:sc}) we have
  $N(\pp)\equiv \pm 1\pmod{2^{s}}$ by \eqref{eq:split2}.
  Furthermore, since $\pp$ is inert in $k(\eta_{2^{s+1}})$, we have 
  $N(\pp)\not\equiv \pm 1\pmod{2^{s+1}}$ again by \eqref{eq:split2}.
  It follows that $N(\pp)\equiv \pm 1+2^s\pmod{2^{s+1}}$.
  Since $T/k$ is quadratic (Lemma \ref{lem:sc}) and $\pp$ is inert in $T$,
  $N(\pp_2)=N(\pp)^2\equiv 1+2^{s+1}\not\equiv 1\pmod{2^{s+2}}$.
  This together with $N(\pp_2)\equiv 1\pmod{2^e}$ shows $e\leq s+1$. 
\end{proof}
 
\begin{proof}[Proof of Theorem~\ref{thm:2p} for $K/k$ exceptional and $p=2$]\mbox{}\\
  Let  $n>b_2(K/k)=\hp{2}(K/k)+s_2(K)+1$.
  Choose $P_1,P_2$ as in Proposition~\ref{prop:nspB}.
  Then, if there exists a $2^n$-cover with full local degree
  in $\set{\PP_1,\PP_2}$ for some $\PP_1\in P_1$ and $\PP_2\in P_2$,
  we can conclude $\hp{2}(K/k)\geq n-s_2(K)-1$.
  Since this contradicts the hypothesis $n>\hp{2}(K/k)+s_2(K)+1$,
  there is no $2^n$-cover of $K/k$ with full local degree in $\set{\PP_1,\PP_2}$
  for any $\PP_1\in P_1, \PP_2\in P_2$.
\end{proof}

This completes the proof of Theorem \ref{thm:x} also for exceptional $K/k$ and even $m$.

\bibliographystyle{plain}
%\bibliography{../master}

\begin{thebibliography}{10}

\bibitem{albert:modern-algebra}
A.~Albert.
\newblock {\em Modern Higher Algebra}.
\newblock Univ. of Chicago Press, Chicago, 1937.

\bibitem{amitsur:central-div-alg}
S.~Amitsur.
\newblock On central division algebras.
\newblock {\em Israel J. Math.}, 12:408--420, 1972.

\bibitem{artin-tate}
E.~Artin and J.~Tate.
\newblock {\em {Class Field Theory}}.
\newblock Benjamin, New York Amsterdam, 1967.

\bibitem{brussel:noncr-prod}
E.~Brussel.
\newblock Noncrossed products and nonabelian crossed products over ${\Q}(t)$
  and ${\Q}((t))$.
\newblock {\em Amer. J. Math.}, 117:377--393, 1995.

\bibitem{brussel:noncr-prod-ff}
E.~Brussel.
\newblock Non-crossed products over function fields.
\newblock {\em Manuscripta Math.}, 107(3):343--353, 2002.

\bibitem{geyer-jensen}
W.-D. Geyer and C.~Jensen.
\newblock Embeddability of quadratic extensions in cyclic extensions.
\newblock {\em Forum Math.}, 19(4):707--725, 2007.

\bibitem{hanke:thesis}
T.~Hanke.
\newblock {\em A Direct Approach to Noncrossed Product Division Algebras}.
\newblock Dissertation, Universit{\"a}t Potsdam, 2001.

\bibitem{hanke:absolute-galois}
T.~Hanke.
\newblock On absolute {G}alois splitting fields of central-simple algebras.
\newblock {\em J. Numb. Th.}, 126:74--86, 2007.

\bibitem{hanke:height}
T.~Hanke.
\newblock The local-global principle for the height of a cyclic extension and
  its special case.
\newblock Preprint, 2009.

\bibitem{kapl:inf-ab-grps}
I.~Kaplansky.
\newblock {\em Infinite abelian groups}.
\newblock University of Michigan Press, Ann Arbor, 1954.

\bibitem{lang:algnt2nd}
S.~Lang.
\newblock {\em Algebraic number theory}, volume 110 of {\em Graduate Texts in
  Mathematics}.
\newblock Springer-Verlag, New York, second edition, 1994.

\bibitem{liedahl:metacyclic-pres}
S.~Liedahl.
\newblock Presentations of metacyclic {$p$}-groups with applications to
  {$K$}-admissibility questions.
\newblock {\em J. Algebra}, 169(3):965--983, 1994.

\bibitem{nakayama:witt-thm}
T.~Nakayama.
\newblock {Divisionsalgebren \"uber diskret bewerteten perfekten K\"orpern.}
\newblock {\em J. Reine Angew. Math.}, 178:11--13, 1937.

\bibitem{neukirch:einb-zt}
J.~Neukirch.
\newblock {\"U}ber das {E}inbettungsproblem der algebraischen {Z}ahlentheorie.
\newblock {\em Inv. math.}, 21:59--116, 1973.

\bibitem{pierce:ass-alg}
R.~Pierce.
\newblock {\em {Associative Algebras}}.
\newblock Springer-Verlag, New York, 1982.

\bibitem{weil:bnt3rd}
A.~Weil.
\newblock {\em Basic number theory}.
\newblock Springer-Verlag, New York, third edition, 1974.
\newblock Die Grundlehren der Mathematischen Wissenschaften, Band 144.

\bibitem{witt:schiefkoerper}
E.~Witt.
\newblock {Schiefk\"orper \"uber diskret bewerteten K\"orpern.}
\newblock {\em J. Reine Angew. Math.}, 176:153--156, 1936.

\end{thebibliography}

\def\cprime{$'$}

\end{document}